\documentclass{amsart}
\usepackage{amsfonts}
\usepackage{color}

\usepackage{amsmath,amssymb,amsthm, epsfig}

\usepackage{hyperref}

\usepackage{amsmath}

\usepackage{amssymb}

\usepackage[mathscr]{eucal}

\def\R{{\mathbb {R}}}

\usepackage{mathtools}
\mathtoolsset{showonlyrefs}

\newlength{\hchng}
\newlength{\vchng}
\def \R {\mathbb{R}}

\def \div {\mathrm{div}}

\def \dist {\mathrm{dist}}

\def \Leb {\mathscr{L}^N}

\newcommand{\defeq}{\mathrel{\mathop:}=}

\newtheorem{theorem}{Theorem}[section]
\newtheorem{lemma}[theorem]{Lemma}

\newtheorem{corollary}[theorem]{Corollary}
\theoremstyle{definition}

\newtheorem{definition}[theorem]{Definition}
\newtheorem{notation}[theorem]{Notation}
\newtheorem{example}[theorem]{Example}
\theoremstyle{remark}
\newtheorem{remark}[theorem]{Remark}
\numberwithin{equation}{section}

\newcommand{\intav}[1]{\mathchoice {\mathop{\vrule width 6pt height 3 pt depth  -2.5pt
\kern -8pt \intop}\nolimits_{\kern -6pt#1}} {\mathop{\vrule width
5pt height 3  pt depth -2.6pt \kern -6pt \intop}\nolimits_{#1}}
{\mathop{\vrule width 5pt height 3 pt depth -2.6pt \kern -6pt
\intop}\nolimits_{#1}} {\mathop{\vrule width 5pt height 3 pt depth
-2.6pt \kern -6pt \intop}\nolimits_{#1}}}

\begin{document}
	
\title[Regularity for $p-$dead core problems and the limiting behaviour as $p \to \infty$]{Regularity
properties for $p-$dead core problems and their asymptotic limit as $p \to \infty$}

\author[J.V. da Silva, J.D. Rossi and A.M. Salort]{Jo\~{a}o V\'{i}tor da Silva, Julio D. Rossi and Ariel M. Salort}

\address{Departamento de Matem\'atica, FCEyN - Universidad de Buenos Aires and
\hfill\break \indent IMAS - CONICET
\hfill\break \indent Ciudad Universitaria, Pabell\'on I (1428) Av. Cantilo s/n. \hfill\break \indent Buenos Aires, Argentina.}

\email[J.V. da Silva]{jdasilva@dm.uba.ar}

\email[J.D. Rossi]{jrossi@dm.uba.ar}
\urladdr{http://mate.dm.uba.ar/~jrossi}

\email[A.M. Salort]{asalort@dm.uba.ar}
\urladdr{http://mate.dm.uba.ar/~asalort}

\begin{abstract}
We study regularity issues and the limiting behavior as $p\to\infty$ of nonnegative solutions for elliptic equations of $p-$Laplacian type ($2 \leq  p< \infty$) with a strong absorption:
\begin{equation*}
      -\Delta_p u(x) + \lambda_0(x) u_{+}^q(x) = 0 \quad \mbox{in} \quad \Omega \subset \R^N,
\end{equation*}
where $\lambda_0>0$ is a bounded function, $\Omega$ is a bounded domain and $0\leq q<p-1$. When $p$ is fixed, such a model is mathematically interesting since it permits the formation of dead core zones, i.e, a priori unknown regions where non-negative solutions vanish identically. First, we turn our attention to establishing sharp quantitative regularity properties for $p-$dead core solutions. Afterwards, assuming that $\displaystyle \ell \defeq \lim_{p \to \infty} q(p)/p \in [0, 1)$ exists, we establish existence for limit solutions as $p\to \infty$, as well as we characterize the corresponding limit operator governing the limit problem. We also establish sharp $C^{\gamma}$ regularity estimates for limit solutions along free boundary points, that is, points on $ \partial \{u>0\} \cap \Omega$  where the sharp regularity exponent is given explicitly by $\gamma = \frac{1}{1-\ell}$. Finally, some weak geometric and measure theoretical properties as non-degeneracy, uniform positive density, porosity and convergence of the free boundaries are proved.
\end{abstract}
\maketitle


\section{Introduction}\label{Introd}

 Quasi-linear elliptic equations with free boundaries appear in a number of phenomena linked through reaction-diffusion and absorption processes in pure and applied sciences. Some remarkable problems are derived from models in chemical-biological processes, combustion phenomena and population dynamics, just to mention a few examples. Regarding these studies, an often more relevant problem is that arising from diffusion processes with sign constrain (the well-known \textit{one-phase problems}), which are in chemical-physical situations the only significant case to be considered (cf. \cite{Aris1}, \cite{Aris2}, \cite{BSS}, \cite{Diaz}, \cite{HM} and references therein for some motivation). A class for such problems is given by
\begin{equation}\label{DCP}
\left\{
\begin{array}{rclcl}
     -\mathcal{Q} u(x) + f(u)\chi_{\{u>0\}} (x) & = & 0 & \mbox{in} & \Omega  \\
     u(x) & = & g(x) & \mbox{on} & \partial \Omega,
\end{array}
\right.
\end{equation}
 where $\mathcal{Q}$ is a quasi-linear elliptic operator in divergence form with $p-$structure for $2\leq p< \infty$ (cf. \cite{Choe1}, \cite{DB0} and \cite{Ser64} for more details), and  $\Omega \subset \R^N$ is a regular and bounded domain. In this context $f$ is a continuous and increasing reaction term satisfying $f(0) = 0$ and $0\leq g \in C^0(\partial \Omega)$. In models from applied sciences, $f(u)$ represents the ratio of the reaction rate at concentration $u$ to the reaction rate at concentration one. Recall that, when the nonlinearity $f \in C^1(\Omega)$ is locally $p-1$ Lipschitz near zero (we say that $f$ satisfies a Lipschitz condition of order $p-1$ at $0$ if there exist constants $\mathfrak{M}, \delta>0$ such that $f(u)\leq \mathfrak{M}u^{p-1}$ for $0<u<\delta$), it follows from the Maximum Principle that nonnegative solutions must be, in fact, strictly positive (cf. \cite{Vaz}). Nevertheless, the function $f$ may fail to be differentiable or even to decay fast enough at the origin. For instance, if $f(t)$ behaves as $t^q$ with $0<q<p-1$, $f$ fails to be Lipschitz of order $p-1$ at the origin; in this case, problem \eqref{DCP} has an absence of Strong Minimum Principle, i.e., nonnegative solutions may vanish completely within an \textit{a priori} unknown region of positive measure $\Omega_0 \subset \Omega$ known as the \textit{Dead Core} set (cf. D\'{i}az's Monograph \cite[Chapter 1]{Diaz} for a survey about this subject).
As an illustration of the previous discussion (cf. \cite{Aris1}, \cite{Aris2}, \cite{FriePhil} and \cite{Stak}), for a   domain $\Omega \subset \R^N$, certain (stationary) isothermal, and irreversible catalytical reaction processes might be mathematically modeled by boundary value problems of reaction-diffusion type of the form
\begin{equation}\label{ExChemCat}
\left\{
\begin{array}{rclcl}
     -\Delta u(x) + \lambda_0 (x) u_{+}^q(x)  & = & 0 & \mbox{in} & \Omega, \\
     u(x) & = & 1 & \mbox{on} & \partial \Omega,
\end{array}
\right.
\end{equation}
where in this context $u$ represents the density of a chemical reagent (or gas) and the non-Lipschitz kinetics corresponds to the $q^{\text{th}}$ order isothermal of Freundlich. Moreover, $\lambda_0>0$ is known as \textit{Thiele Modulus} and it controls the ratio of the reaction rate to the diffusion-convection rate. As before, when $q \in (0, 1)$,  the strong absorption due to chemical reaction may be faster than the supply caused by diffusion across the boundary, which can lead the chemical reagent to vanish in some subregions (the dead-cores), namely $\Omega_0 \defeq \{x \in \Omega: u(x) =0\} \subset \Omega$. In such zones no chemical reaction takes place. For this reason, the knowledge of the qualitative/quantitative behavior of the dead-core solutions plays a key role in chemical engineering and other applied fields.
\begin{figure}[ht]
\begin{center}
\includegraphics[width=5.1cm,height=4.8cm]{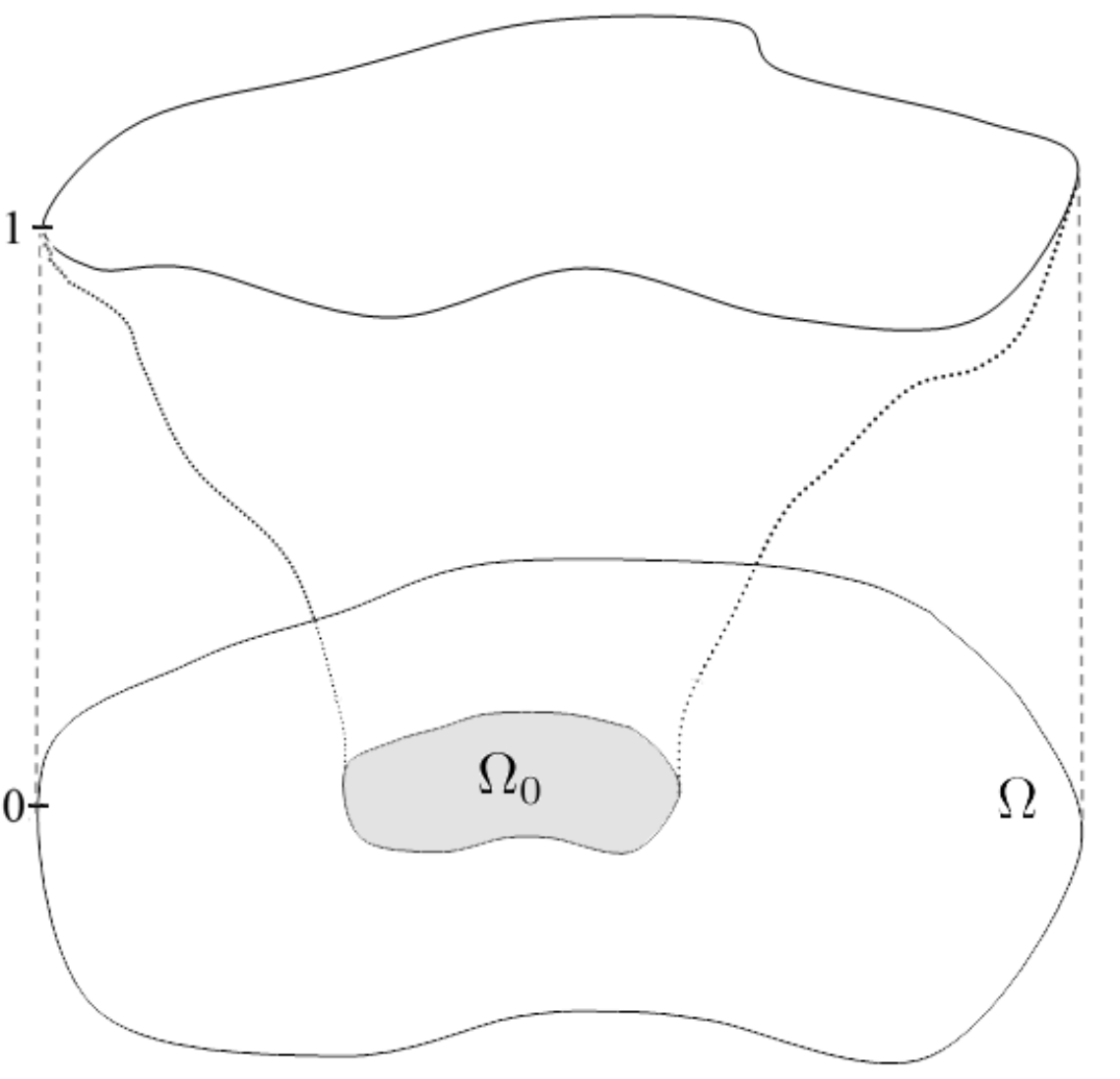}
\caption{The dead-core set $\Omega_0$ illustrating the isothermal and irreversible catalytical reaction process from \eqref{ExChemCat}.}
\label{dib3}
\end{center}
\end{figure}

 Dead core type problems have received growing attention throughout the last four decades, including: existence of solutions, formation of dead core sets, properties of localization, effectiveness factors among others. A summarized literature can be found in the seminal works of Bandle \textit{et al} \cite{BSS, BV}, D\'{i}az \textit{et al} \cite{Diaz}, \cite{DiazHern}, \cite{DiazHerr}, \cite{DiazVer}, Pucci-Serrin \cite{PS06} and references therein. In spite of the fact that there is a large amount of literature on dead core problems, quantitative regularity properties for such models with $p-$Laplacian type structure are far less studied (compare with da Silva \textit{et al} in \cite{OSS} and \cite{SS} for instance). This fact is one of our starting points for the current study.

 In this article we study diffusion problems governed by the $p-$Laplacian operator for which a Minimum Principle is not available. Particularly, we are interested in prototypes coming from combustion problems, chemical models (like catalytical processes in chemical engineering) or enzymatic kinetics where the existence of dead cores plays an important role in the understanding of such a chemical-physical models (cf. \cite{Aris1}, \cite{Aris2}, \cite{Guo} and \cite{HM}). The  model to be analyzed here is given by
\begin{equation}\label{Eqp-Lapla}
\left\{
\begin{array}{rclcl}
     -\Delta_p u_p(x)+\lambda_0(x) (u_p)_{+}^q(x) & = & 0 & \mbox{in} & \Omega, \\
     u_p(x) & = & g(x) & \mbox{on} & \partial \Omega.
\end{array}
\right.
\end{equation}
Here $\Delta_p u = \div(|\nabla u|^{p-2}\nabla u)$ stands for the $p$-Laplace operator, $\lambda_0>0$ is a bounded
function (bounded away from zero and from infinity) and $0\leq q < p-1$ is the \textit{order of reaction}.
The boundary datum is assumed to be continuous and nonnegative, $0\leq g \in C^0(\partial \Omega)$.
In this context, \eqref{Eqp-Lapla} is said to be an equation with \textit{strong absorption} condition and $\partial\{u>0\}\cap \Omega$ is (the physical) \emph{free boundary} of the problem.

Observe also that the unique weak solution (cf. \cite[Theorem 1.1 ]{Diaz}) to \eqref{Eqp-Lapla} appears when considering minimizers of the following (zero) constrained nonlinear $p-$obstacle type problem
\begin{equation}
\label{eqMinHaus}
    \min\left\{\int_{\Omega} \left(\frac{1}{p}|\nabla v (x)|^p + \lambda_0(x)\frac{v^{q+1}}{q+1}\chi_{\{v>0\}}(x)\right)dx\colon v \in W^{1, p}(\Omega), \,\, v\geq 0 \,\, \text{and}\,\, v=g \,\, \text{on} \,\,\partial \Omega\right\}.
\end{equation}
 Variational problems like \eqref{eqMinHaus} were widely developed in the last decades, see \cite{KKPS}, \cite{LeeShah} and \cite{Manf88} for some examples on this subject.

The study of \eqref{Eqp-Lapla} is meaningful, not only for applications, but also for its relation with several mathematical free boundary problems appearing in the literature  (cf. Alt-Phillips \cite{AP}, D\'{i}az \cite{Diaz}, Friedman-Phillips \cite{FriePhil} and Phillips \cite{Phil}, Shahgholian \textit{et al} \cite{KKPS} and \cite{LeeShah} for problems with variational structure and da Silva \textit{et al} \cite{LRS} and Teixeira \cite{Tei-16} for a non-variational counterpart; we
also refer the reader to da Silva \textit{et al} \cite{OS} and \cite{OSS} for similar problems in parabolic settings).

 A fundamental issue appearing in free boundary problems consists in inferring which is the optimal expected regularity to weak solutions. For example, if we fix $0<q<p-1$ and $R>0$, the one dimensional profile $u: (-R, R) \to \R_{+}$ given by
$$
    u(x) = \mathfrak{c}  x_{+}^{\frac{p}{p-1-q}}
$$
(for an appropriate constant $\mathfrak{c} = \mathfrak{c}(p, q)>0$) is a weak solution to
$ - (|u'|^{p-2} u')' +u^{q} = 0$ in $(-R, R)$.
It is known that in general, solutions to \eqref{Eqp-Lapla} are  $C_{\text{loc}}^{1, \gamma}$ for some $\gamma \in (0, 1)$ (cf. \cite{Choe1}, \cite{DB0} and \cite{Tolk} for more details). However, for such an example, fixed $p>1$, one observes that $u \in C_{\text{loc}}^{\lfloor \alpha \rfloor, \beta}$, where
$$
  \alpha(p, q) \defeq \frac{p}{p-1-q} \quad \text{and} \quad \beta(p, q) \defeq \frac{p}{p-1-q} - \left\lfloor \frac{p}{p-1-q} \right\rfloor.
$$
Here $\left\lfloor \cdot \right\rfloor$ stands for the integer part. Note that when $q>0$,
\begin{eqnarray*}
 \alpha(p, q) = \frac{p}{p-q-1} > \frac{p}{p-1}\quad \Rightarrow \quad \text{Improved regularity estimates along the free boundary}.
\end{eqnarray*}
In effect, $C_{\text{loc}}^{\alpha^{\sharp}}$ regularity for the exponent $\alpha^{\sharp} \defeq \frac{p}{p-1} = 1+\frac{1}{p-1}$ is the better regularity estimate expected (and proved in some cases) for weak solutions of the  $p-$Laplacian operator with bounded RHS, see \cite{ATU2} for a systematic treatment of this subject (the so termed $C^{p^{\prime}}$ regularity conjecture) and compare with \cite{ALS} in the context of $p-$obstacle problems. Furthermore, notice that $\alpha(p, q)> 2$ provided that $q > \max\left\{0,  \frac{p-2}{2}\right\}$, which means that it is a classical solution, even across the free boundary. The proper understanding of this phenomenon should yield decisive geometric information about the solution and its free boundary, and this is the main goal (of the first part) of our investigation. Therefore, we will show that \textit{any} solution to \eqref{Eqp-Lapla} behaves (near its free boundary) like the previous one-dimensional example.

First, we prove which is the growth rate of  solutions leaving their free boundaries. Throughout this manuscript universal constants are the ones depending only on structural and physical parameters of the problem, i.e., $N, p, q, \Omega^{\prime}$ and the bounds for $\lambda_0$.

\begin{theorem}[{\bf Strong non-degeneracy}]\label{LGR} Let $u$ be a nonnegative, bounded weak solution to \eqref{Eqp-Lapla}, $\Omega^{\prime} \Subset \Omega$ and let $x_0 \in \overline{\{u >0\}} \cap \Omega^{\prime}$. Then there exists a universal constant $\mathfrak{C}_0>0$ such that for all $0<r<\min\{1, \dist(\Omega^{\prime}, \partial \Omega)\}$ there holds
\begin{equation}\label{nondeg.est}
   \displaystyle \sup_{\partial B_r(x_0)} \,u(x) \geq \mathfrak{C}_0\left(N, p, q, \inf_{\Omega} \lambda_0(x)\right)
   r^{\frac{p}{p-1-q}}.
\end{equation}
\end{theorem}

Next, we prove the following improved regularity estimate at free boundary points:

\begin{theorem}[{\bf Improved regularity along the free boundary}]\label{ThmGR1} Let $u$ be a nonnegative, bounded weak solution to \eqref{Eqp-Lapla}, $\Omega^{\prime} \Subset \Omega$ and $x_0 \in \partial \{u >0\} \cap \Omega^{\prime}$. Then, there exists a universal constant $\mathfrak{C}_1= \mathfrak{C}_1\left(N, p, q, \inf_{\Omega} \lambda_0(x)\right)>0$ such that
\begin{equation}\label{ImpEst}
    u(x) \leq \mathfrak{C}_1
    \|u\|_{L^{\infty}(\Omega)}  |x-x_0|^{\frac{p}{p-1-q}}
\end{equation}
for all $0<|x-x_0|\ll \min\{1, \dist(\Omega^{\prime}, \partial \Omega)\}$.
\end{theorem}

 The insight for proving Theorem \ref{ThmGR1} comes from a (finer) asymptotic blow-up analysis combined with a sharp control of the growth rate of solutions close to the free boundary, and it was inspired in \cite{KKPS}, \cite{LeeShah}, \cite{Tei16} and \cite{Tei-16}. Heuristically, for an appropriate family of normalized and scaled solutions, if the ``magnitude'' of the \textit{Thiele modulus} is uniformly controlled (with small enough bounds), then we must expect the following iterative geometric decay
 $$
    \displaystyle \sup_{B_{\frac{1}{2^{k}}}(x_0)} u(x) \leq \mathfrak{C}_{\sharp}\left(N, p, q, \inf_{\Omega} \lambda_0(x)\right) 2^{-\frac{kp}{p-1-q}}
    \qquad \forall \,\,x_0 \in \partial\{u>0\}\cap \Omega^{\prime} \quad \text{and} \quad \forall k \in \mathbb{N},
 $$
which will yield the aimed geometric regularity estimate along free boundary points (See Section \ref{SecImpReg}).

Observe that, using our regularity estimates from Theorem \ref{ThmGR1} and Theorem \ref{LGR}, we are able to show that any weak solution to \eqref{Eqp-Lapla} is ``confined'' between the graph of two suitable multiples of $\dist(\cdot, \partial \{u>0\})^{\frac{p}{p-1-q}}$, i.e.,
$$
  \mathfrak{c}_{\sharp} [\dist(x_0, \partial \{u>0\})]^{\frac{p}{p-1-q}} \leq u(x_0) \leq \mathfrak{c}^{\sharp}
   [\dist(x_0, \partial \{u>0\})]^{\frac{p}{p-1-q}},
$$
where $x_0 \in \{u>0\}$ is  close enough to the free boundary and $\mathfrak{c}_{\sharp}, \mathfrak{c}^{\sharp}>0$ are universal constants (depending on $N$, $p$, $q$, $\lambda_0$ and the uniform bound for $u$), see Corollaries \ref{CorNonDeg} and \ref{DistEst} for more details.

Recall that our model \eqref{Eqp-Lapla} involves an ``ellipticity factor'' deteriorating on an unknown set, the \textit{free boundary}. Particularly, this phenomenon implies less diffusivity for the model near such a set and consequently, regularity properties of solutions become more delicate to obtain. In our case, such phenomenon occurs along the set of critical points of solutions, i.e.,
$$
     \mathcal{C}_\Omega[u] \defeq \left\{x \in \Omega: |\nabla u(x)|=0\right\}.
$$
Since we can re-write our prototype operator as $$\displaystyle \Delta_p u = |\nabla u|^{p-2}\Delta u + (p-2)
|\nabla u|^{p-4} \sum_{i,j=1}^{N} \frac{\partial u}{\partial x_i} \frac{\partial u}{\partial x_j}
\frac{\partial^2 u}{\partial x_i \partial x_j},$$
(here we are talking about viscosity solutions)
solutions to \eqref{Eqp-Lapla} can only be ``irregular'' if $|\nabla u|$ is small enough, because the operator degenerates as $|\nabla u|\approx 0$. Surprisingly, for the model \eqref{Eqp-Lapla} we can control (uniformly) the decay of the gradient at free boundary points via a sophisticated iterative mechanism. In a precise manner (see Lemma \ref{IRresult2.66} for the details), for any weak solution to \eqref{Eqp-Lapla} it holds that
\begin{enumerate}
  \item[(i)] $\mathcal{C}_\Omega[u] \subset \overline{ \{u=0\}} \cap \Omega$,
  \item[(ii)] $\displaystyle \sup_{B_r(x_0)} |\nabla u| \leq \mathfrak{c}_{\ast}r^{\frac{1+q}{p-1-q}}$, \qquad $ \forall x_0 \in \partial \{u>0\} \cap \Omega^{\prime}$ and $r\ll 1$, for a universal constant $\mathfrak{c}_{\ast}>0.$
\end{enumerate}

The approach employed in our article is flexible enough to yield several other interesting applications concerning regularity at free boundary points. Among these we establish a new estimate in the $L^2-$average sense for $p-$dead core solutions. More precisely (see Lemma \ref{LemmL2Est}), there exists a constant $M = M(N, p, q, \lambda_0)$ such that, for any weak solution to \eqref{Eqp-Lapla} and any $x_0$ interior free boundary point, there holds
$$
   \displaystyle \left(\intav{B_r(x_0)} (|\nabla u(x)|^{p-2}|D^2 u(x)|)^2 dx\right)^{\frac{1}{2}} \leq M r^{\frac{pq}{p-1-q}} \quad \forall \,\,0<r \ll1.
$$

The second part of our article is focused in analyzing the asymptotic behavior as $p$ and $q$ diverge.
Here we will assume that the boundary condition is Lipschitz continuous, $g\in Lip (\partial \Omega)$,
and note that it can be extended as a Lipschitz function to the whole $\overline{\Omega}$ with the same
Lipschitz constant, see \cite{ACJ}.  Recently, motivated by game theory (``Tug of-war games''), in \cite{JPR} it is studied the following problem
\begin{equation*}
  \left\{
  \begin{array}{rclcl}
    \displaystyle  \Delta_p \,u_p(x)& = & f(x) & \mbox{in}& \Omega, \\
    u_p(x) & = & g(x) & \mbox{on} & \partial \Omega,
  \end{array}
  \right.
\end{equation*}
   with a forcing term $f\geq0$. In this context, $\{u_p\}$ converges, up to a subsequence, to a limiting function $u_{\infty}$, which fulfills the following problem in the viscosity sense
\begin{equation}\label{eqlimp-Lap2}
  \left\{
  \begin{array}{rclcl}
    \min\left\{ \Delta_{\infty} \, u_{\infty}(x), |\nabla u_{\infty}(x)|- \chi_{\{f>0\}} (x) \right\} & = & 0 & \mbox{in}&  \Omega,\\
    u_{\infty}(x) & = & g(x) & \mbox{on} & \partial \Omega,
  \end{array}
  \right.
\end{equation}
where $\displaystyle \Delta_\infty u(x) \defeq \nabla u(x)^TD^2u(x) \nabla u(x)$ is the well-known \textit{$\infty-$Laplace operator} (cf. \cite{ACJ} for a survey).
At this point, a natural question is the expected behavior of the $p-$dead core solutions and their free boundaries as $p\to\infty$. That is precisely the subject of the second part of this manuscript:  we turn our attention to the study of several geometric and analytical properties for limit solutions and their free boundaries. Motivated by the analysis of the asymptotic behavior of several variational problems (see for example \cite{daSRS1, daSRS2, JPR, MRU, RT, RTU, RW}), we will focus our attention on the analysis under the condition that the diffusivity degree of the operator and the reaction factor of the model are large enough. Intuitively, we would like to understand and characterize the physical process when $p$ and $q$ diverge in a controlled manner. We will also assume in this limit procedure that the boundary datum $g$ is a fixed Lipschitz function.

In our first result, we prove the convergence of the family of $p-$dead core solutions, as well as we deduce the corresponding limit operator (in non-divergence form) driving the limit equation.

\begin{theorem}[{\bf Limiting equation}]\label{MThmLim1} Let $(u_p)_{p\geq 2}$ be the family of solutions to
\eqref{Eqp-Lapla} with $g\in Lip (\partial \Omega)$.
Assume that $\displaystyle \ell \defeq \lim_{p \to \infty} q(p)/p \in [0, 1)$ exists. Then, up to a subsequence, $u_p \to u_{\infty}$ uniformly in $\overline{\Omega}$.
Furthermore, such a limit fulfills
\begin{equation}\label{EqLim}
\left\{
\begin{array}{rcrcl}
  \max\left\{-\Delta_{\infty} u_{\infty}, \,\, -|\nabla u_{\infty}| + u_{\infty}^{\ell}\right\} & = & 0 & \text{in} & \{u_{\infty} > 0\} \cap \Omega ,\\
  -\Delta_{\infty} u_{\infty} & =& 0& \text{in} & \{u_{\infty} = 0\} \cap \Omega,\\
  u_{\infty} & = & g & \text{on} & \partial \Omega,
\end{array}
\right.
\end{equation}
in the viscosity sense.
\end{theorem}

We also obtain improved regularity estimates along a free boundary point for the limit profiles.

\begin{theorem}[{\bf Sharp growth for limit solutions}]\label{MThmLim3}
Assume that $\displaystyle \ell \defeq \lim_{p \to \infty} q(p)/p \in [0, 1)$ exists. Let $u_{\infty}$ be a uniform limit of the family of solutions $u_p$ of \eqref{Eqp-Lapla} with $g\in Lip (\partial \Omega)$ and $\Omega^{\prime} \Subset \Omega$. Then, for any $x_0 \in \partial \{u_{\infty}>0\} \cap \Omega^{\prime}$ and $0<r \ll 1$ the following estimate holds:
\begin{equation}\label{estim}
  \displaystyle \sup_{B_r(x_0)} u_{\infty}(x) \leq 2\cdot2^{\frac{1}{1-\ell}}(1-\ell)^{\frac{1}{1-\ell}}r^{\frac{1}{1-\ell}}.
\end{equation}
\end{theorem}

Finally, we obtain a sharp lower control on how limit solutions detach from their free boundaries.

\begin{theorem}[{\bf Strong non-degeneracy for limit solutions}]\label{MThmLim2}
Assume that $\displaystyle \ell \defeq \lim_{p \to \infty} q(p)/p \in [0, 1)$ exists. Let $u_{\infty}$ be a uniform limit to solutions $u_p$ of \eqref{Eqp-Lapla}  with $g\in Lip (\partial \Omega)$ and $\Omega^{\prime} \Subset \Omega$. Then, for any
$x_0 \in \partial \{u_{\infty}>0\} \cap \Omega^{\prime}$ and any $0<r \ll 1$, the following estimate holds:
\begin{equation}\label{estim.22}
  \displaystyle \sup_{B_r(x_0)} u_{\infty}(x) \geq (1-\ell)^{\frac{1}{1-\ell}} r^{\frac{1}{1-\ell}}.
\end{equation}
\end{theorem}

In contrast with the previous results, we do not expect (in general) a point-wise gradient estimate for limit solutions. However, we are able to prove the following estimate in average for the gradient.

\begin{theorem}[{\bf An average gradient estimate for limit solutions}]\label{ThmGradLim}
Assume that $\displaystyle \ell \defeq \lim_{p \to \infty} \frac{q(p)}{p} \in [0, 1)$ exists. Let $u_{\infty}$ be a uniform limit to solutions $u_p$ of \eqref{Eqp-Lapla}  with $g\in Lip (\partial \Omega)$ and $\Omega^{\prime} \Subset \Omega$. Then, for any $x_0 \in \partial \{u_{\infty}>0\} \cap \Omega^{\prime}$ and any $0<r \ll 1$, the following estimate holds:
\begin{equation}\label{estim.33}
  \displaystyle \intav{B_r(x_0)} |\nabla u_{\infty}(x)|dx \leq 2\cdot2^{\frac{\ell}{1-\ell}}(1-\ell)^{\frac{1}{1-\ell}} r^{\frac{1}{1-\ell}}.
\end{equation}
\end{theorem}

We also present several other quantitative/qualitative properties for free boundaries of limit solutions, among which we include a necessary condition in order to have convergence of the free boundaries
$$
  \partial \{u_p > 0\} \to \partial \{u_{\infty} > 0\}\quad \mbox{as} \quad  p\to \infty,
$$
in the sense of the Hausdorff distance (see Theorem \ref{MThmLim5}).
Another important point which we deal with concerns regularity of certain limit free boundary points under an appropriate geometric condition (see Theorem \ref{RegFB}). At the end we put forward some examples in order to illustrate some features of the limit problem.

\section{Preliminaries}\label{Section2}
In this section we introduce some preliminary results that we will use throughout the article.

\begin{notation}
We adopt the following notations:
\begin{itemize}
\item[\checkmark] $u_{+} = \max\{u, 0\}$.
\item[\checkmark] $\mathfrak{F}(u_0, \Omega) \defeq \partial \{u_0 >0\} \cap \Omega$ will mean the free boundary.
\item[\checkmark] $\Leb$ denotes the $N$-dimensional Lebesgue measure.
\item[\checkmark] $\displaystyle \mathcal{S}_r[u](x_0) \defeq   \sup_{B_r(x_0)} u(x) \quad \text{and} \quad  \displaystyle \mathcal{I}_r[u](x_0) \defeq   \inf_{B_r(x_0)} u(x).$ The center of the ball is omitted when $x_0 = 0$.
\end{itemize}
\end{notation}

\begin{definition}[{\bf Weak solution}] $u \in W^{1, p}_{\text{loc}}(\Omega)$ is a weak super-solution (resp. sub-solution) to
\begin{equation}\label{eqweaksol}
     -\Delta_p u = \Psi(x, u) \quad \text{in }  \Omega,
\end{equation}
if for all $0\leq \varphi \in C^1_0(\Omega)$ it holds
$$
  \displaystyle \int_{\Omega} |\nabla u|^{p-2}\nabla u\cdot \nabla \varphi(x)\,dx \geq \int_{\Omega} \Psi(x, u)\varphi(x)\,dx \quad \left(\text{resp.}\,\,  \leq \int_{\Omega} \Psi(x, u)\,dx\right).
$$
Finally, $u$ is a weak solution to \eqref{eqweaksol} when it is simultaneously a super-solution and a  sub-solution.
\end{definition}

Since we are assuming that $p \geq 2$, then \eqref{Eqp-Lapla} is not singular at points where the gradient vanishes. Consequently, the mapping $x \mapsto \Delta_p \phi(x) = |\nabla \phi(x)|^{p-2}\Delta \phi(x) + (p-2)|\nabla \phi(x)|^{p-4}\Delta_{\infty} \phi(x)$ is well-defined and continuous for all $\phi \in C^2(\Omega)$.

Next, we introduce the notion of viscosity solution to \eqref{Eqp-Lapla}. We refer to the survey \cite{CIL} for the general theory of viscosity solutions.

\begin{definition}[{\bf Viscosity solution}]
  An upper (resp. lower) semi-continuous function $u: \Omega \to \R$ is said to be a viscosity sub-solution (resp. super-solution) to \eqref{Eqp-Lapla} if, whenever $x_0 \in \Omega$ and $\phi \in C^2(\Omega)$ are such that $u-\phi$ has a strict local maximum (resp. minimum) at $x_0$, then
$$
   -\Delta_p \phi(x_0) + \lambda_0(x)\phi_{+}^q(x_0)\leq 0 \quad (\text{resp.} \,\,\,\geq 0).
$$
Finally, $u \in C(\Omega)$ is said to be a viscosity solution to \eqref{Eqp-Lapla} if it is
simultaneously a viscosity sub-solution and a viscosity super-solution.
\end{definition}

\begin{definition}[{\bf Viscosity solution for the limit equation}]\label{DefVSlimeq} A non-negative function $u \in C(\Omega)$ is said to be a viscosity solution to \eqref{EqLim} if:
\begin{enumerate}
  \item it is a viscosity sub-solution, that is, whenever $x_0 \in \Omega$ and $\phi \in C^2(\Omega)$ are such that $u(x_0) = \phi(x_0)$ and
$u(x)<\phi(x)$, when $x \neq x_0$, then
$$
   -\Delta_{\infty} \phi(x_0) \leq 0 \quad \text{and} \quad -|\nabla \phi(x_0)|+\phi_+^{\ell}(x_0) \leq 0;
$$
  \item it is a viscosity super-solution, that is, whenever $x_0 \in \Omega$ and $\phi \in C^2(\Omega)$ are such that $u(x_0) = \phi(x_0)$ and
$u(x)>\phi(x)$, when $x \neq x_0$, then
$$
   -\Delta_{\infty} \phi(x_0) \geq 0 \quad \text{or} \quad -|\nabla \phi(x_0)|+\phi_+^{\ell}(x_0) \geq 0.
$$
\end{enumerate}
\end{definition}

The following Harnack inequality will be useful for our approach.

\begin{theorem}[{{\bf Serrin's Harnack inequality} \cite[Theorem 6 and Theorem 9]{Ser64}}]\label{harnack}
Let $f \in L^{\infty}(B_1)$ and $u$ be a non-negative weak solution to $ -\Delta_p u (x) = f(x) \quad \textit{in} \quad B_1$. Then there exists a constant $C= C(N, p)>0$ such that
\begin{equation*}
	  \displaystyle \mathcal{S}_{\frac{1}{2}}[u] \leq C(N, p)\left( \mathcal{I}_{\frac{1}{2}}[u]+\|f\|_{L^{\infty}(B_1)}^{\frac{1}{p-1}}\right).
\end{equation*}
\end{theorem}

Let us recall an important inequality from Nonlinear Analysis.

\begin{theorem}[{\bf Morrey's inequality}]\label{MorIneq} Let $N<p\leq \infty$. Then for $u \in W^{1, p}(\Omega)$, there exists a constant $\mathfrak{C}(N, p)>0$ such that
$$
  \|u\|_{C^{0, 1-\frac{N}{p}}(\Omega)} \leq \mathfrak{C}(N, p)\|\nabla u\|_{L^{p}(\Omega)}.
$$
\end{theorem}

We must highlight that the dependence of $\mathfrak{C}$ on $p$ does not deteriorate as $p \to \infty$. In fact,
$$
  \mathfrak{C}(N, p) \defeq \frac{2\mathfrak{c}(N)}{|\partial B_1|^{\frac{1}{p}}}\left(\frac{p-1}{p-N}\right)^{\frac{p-1}{p}},
$$
where $\mathfrak{c}(N)>0$ is a constant that depends only on the dimension.

We  also use the following  comparison result for weak solutions (see  \cite[Theorem 1.1]{Diaz}).

\begin{lemma}[{\bf Comparison Principle}]\label{comparison} Let $\lambda_0 \in L^{\infty}(\Omega)$ and $f \in C([0, \infty))$ a non-negative and non-decreasing function. Assume that
$$
   -\Delta_p u + \lambda_0(x)f(u)\chi_{\{u>0\}} \geq 0\geq -\Delta_p v  + \lambda_0(x)f(v)\chi_{\{v>0\}} \quad \mbox{in} \quad \Omega
$$
in the weak sense. If $v \leq u $ in $\partial \Omega$ then $v \leq u$ in $\Omega$.
\end{lemma}

The following lemma gives a relation between weak and viscosity sub and super-solutions to \eqref{Eqp-Lapla}.

\begin{lemma}\label{EquivSols} A continuous weak sub-solution (resp. super-solution) $u \in W_{\text{loc}}^{1,p}(\Omega)$ to \eqref{Eqp-Lapla} is a viscosity sub-solution (resp. super-solution) to
$$
  -\left[ |\nabla u(x)|^{p-2} \Delta u(x) + (p-2)|\nabla u(x)|^{p-4}\Delta_{\infty} u\right] = -\lambda_0(x)u_{+}^{q}(x) \quad \text{in} \quad \Omega.
$$
\end{lemma}

\begin{proof} Let us proceed for the case of super-solutions. Fix $x_0 \in \Omega$ and $\phi \in C^2(\Omega)$ such that $\phi$ touches $u$ from below, i.e., $u(x_0) = \phi(x_0)$ and $u(x)> \phi(x)$ for $x \neq x_0$. Our goal is to show that
$$
  -\left[|\nabla \phi(x_0)|^{p-2}\Delta \phi(x_0) + (p-2)|\nabla \phi(x_0)|^{p-4}\Delta_{\infty} \phi(x_0)\right] + \lambda_0(x_0)\phi_{+}^{q}(x_0) \geq 0.
$$
Let us suppose, for the sake of contradiction, that the inequality does not hold. Then, by continuity there exists  $r>0$ small enough such that
$$
   -\left[ |\nabla \phi(x)|^{p-2}\Delta \phi(x) + (p-2)|\nabla \phi(x)|^{p-4}\Delta_{\infty} \phi(x)\right] +\lambda_0(x)\phi_{+}^{q}(x) < 0,
$$
provided that $x \in B_r(x_0)$. Now, we define the function
$$
   \Psi(x) \defeq \phi(x)+ \frac{1}{100}\mathfrak{m}, \quad \text{ where } \quad \mathfrak{m} \defeq \inf_{\partial B_r(x_0)} (u(x)-\phi(x)).
$$
Notice that $\Psi$ verifies $\Psi < u$ on $\partial B_r(x_0)$, $\Psi(x_0)> u(x_0)$ and
\begin{equation}\label{EqPsi}
 -\Delta_p \Psi(x) < -\lambda_0(x)\phi_{+}^{q}(x).
\end{equation}
By extending by zero outside  $B_r(x_0)$, we may use $(\Psi-u)_{+}$ as a test function in \eqref{Eqp-Lapla}. Moreover, since $u$ is a weak super-solution, we obtain
\begin{equation}\label{Eq3.4}
  \displaystyle \int_{\{\Psi>u\}} |\nabla u|^{p-2}\nabla u \cdot \nabla (\Psi-u) dx \geq  -\int_{\{\Psi>u\}} \lambda_0(x_0)u_{+}^{q}(x)(\Psi-u) dx.
\end{equation}
On the other hand, multiplying \eqref{EqPsi} by $\Psi- u$ and integrating by parts we get
\begin{equation}\label{Eq3.5}
  \displaystyle \int_{\{\Psi>u\}} |\nabla \Psi|^{p-2}\nabla \Psi \cdot \nabla (\Psi-u) dx <  -\int_{\{\psi>u\}} \lambda_0(x)\phi_{+}^{q}(x)(\Psi-u) dx.
\end{equation}
Next, subtracting \eqref{Eq3.4} from \eqref{Eq3.5} we obtain
\begin{equation} \label{E2.6}
\displaystyle  \int\limits_{\{\Psi>u\}} \left(|\nabla \Psi|^{p-2}\nabla \Psi - |\nabla u|^{p-2}\nabla u\right) \cdot \nabla (\Psi-u) dx <  \int\limits_{\{\psi>u\}} \lambda_0(x)\left(\phi_{+}^{q}(x)-u_{+}^{q}(x)\right)(\Psi-u)dx <0.
\end{equation}
This implies that
$$\int\limits_{\{\Psi>u\}} |\nabla(\Psi - u)|^p dx < 0.$$
This is a contradiction that proves the desired result.

Similarly, one can prove that a continuous weak sub-solution is a viscosity sub-solution.
\end{proof}

\section{Regularity properties for fixed $(p,q)$}

\subsection{Non-degeneracy of solutions}\label{SecNond}

This section is devoted to prove a weak geometrical property which plays a key role in the description of solutions to dead core type problems, namely the \emph{non-degeneracy} of solutions.

\begin{proof}[{\bf Proof of Theorem \ref{LGR}}]
Notice that, due to the continuity of solutions, we have to prove the estimate just at points in $\{u>0\} \cap \Omega$.

First of all, let us consider the scaled function:
$$u_r(x) \defeq \frac{u(x_0+rx)}{r^{\frac{p}{p-1-q}}}.$$
It is easy to show that it is a weak sub-solution to the equation
$$
    -\Delta_p u_r(x) + \hat{\lambda}_0(x)(u_r)_{+}^{q}(x) = 0 \quad \text{in} \quad B_1,
$$
with $\hat\lambda_0(x) \defeq \lambda_0(x_0 + rx)$.

Now, let us introduce the auxiliary barrier function $\displaystyle \Psi (x) \defeq \mathfrak{C}_0 |x|^{\frac{p}{p-1-q}}$ for a positive constant $\mathfrak{C}_0$ given by
$$
  \displaystyle \mathfrak{C}_0 \defeq \left[\inf_{\Omega}\lambda_0(x) \frac{ \left(p-1-q \right)^{p} }{ p^{p-1}(pq+N(p-1-q) )}\right]^{\frac{1}{p-1-q}}.
$$
A straightforward computation shows that
$$
     -\Delta_p \Psi (x)+ \hat{\lambda}_0\left( x  \right)\Psi^{q}(x) \geq 0  \quad \text{in} \quad B_1 \quad (pointwisely).
$$

Finally, if $u_r \leq \Psi$ on the whole boundary of $B_1$, then the Comparison Principle (Lemma \ref{comparison}) implies that
$$
   u_r \leq \Psi \quad \mbox{in} \quad B_1,
$$
which  contradicts the assumption that $u_r(0)>0$. Therefore, there exists a point $y \in \partial B_1$ such that
$$
      u_r(y) > \Psi(y) = \mathfrak{C}_0,
$$
and scaling back we finish the proof of the theorem.
\end{proof}

As a consequence of  Theorem \ref{LGR} we obtain the following  growth near the free boundary:

\begin{corollary}[{\bf Sharp growth}]\label{CorNonDeg}
Let $u$ be a nonnegative, bounded weak solution to \eqref{Eqp-Lapla} in $\Omega$ and   $\Omega^{\prime} \Subset \Omega$. Then, there exists a universal constant $\mathfrak{C}_{\sharp}>0$ such that
$$
  u(x_0) \geq \mathfrak{C}_{\sharp} [\dist(x_0, \partial \{u>0\})]^{\frac{p}{p-1-q}}
$$
for any $x_0 \in \{u>0\} \cap \Omega^{\prime}$.
\end{corollary}
\begin{proof}
  Suppose that $\mathfrak{C}_{\sharp}$ does not exist. Then there exist a sequence $x_k \in \{u>0\} \cap \Omega^{\prime}$ with
$$
d_k \defeq \dist(x_k, \partial \{u>0\} \cap \Omega^{\prime}) \to 0 \quad \text{as} \quad k \to \infty \quad \text{and} \quad u(x_k) \leq  k^{-1} d_k^{\frac{p}{p-1-q}}.
$$
Now, let us define the auxiliary function $v_k:B_1 \to \R$ by
$$
   v_k(y) \defeq \frac{u(x_k+d_ky)}{d_k^{\frac{p}{p-1-q}}}.
$$
It is easy to check that
\begin{itemize}
  \item[\checkmark] $v_k \geq 0$ in $B_1$.
  \item[\checkmark] $-\Delta_p v_k + \lambda_0(x_k+d_ky)v_k^q = 0$ in $B_1$ in the weak sense.
  \item[\checkmark] $v_k(y) \leq \mathfrak{C}(N, p) d_k^{1-\frac{N}{p}} + \frac{1}{k} \,\, \forall\,\, y \in B_1$ according to local H\"{o}lder regularity of weak solutions, see Theorem \ref{MorIneq}.
\end{itemize}
From the Non-degeneracy Theorem \ref{LGR} and the last sentence we obtain that
\begin{equation}\label{estim.44}
  \displaystyle 0<\mathfrak{C}_0 \left(\frac{1}{2}\right)^{\frac{p}{p-1-q}} \leq \sup_{B_{\frac{1}{2}}} v_k(y) \leq \max\{1, \mathfrak{C}(N, p)\} \left(d_k^{1-\frac{N}{p}}+\frac{1}{k}\right) \to 0 \quad \text{as} \quad k \to \infty,
\end{equation}
which clearly yields a contradiction. This concludes the proof.
\end{proof}

The Non-degeneracy estimate from Corollary \ref{CorNonDeg} implies in particular a non-degeneracy property in measure. As it was commented in the introduction, this estimate is useful in several qualitative contexts of the theory of free boundary problems. We refer to \cite[Theorem 4.4]{SS} for a proof.

\begin{theorem}[{\bf Non-degeneracy in measure}]\label{ThmNDinM} Let $u$ be a weak solution to \eqref{Eqp-Lapla}. Given $\Omega^{\prime} \subset \Omega$ there exist $\rho_0>0$ and $\kappa>0$ depending only on $\Omega^{\prime}$ and universal parameters such that
$$
  \mathcal{L}^N\left(\Omega^{\prime} \cap \{0<u(x)<\rho^{\frac{p}{p-1-q}}\}\right)\leq \kappa\rho \qquad \text{ for any } \quad \rho\leq \rho_0.
$$
\end{theorem}

\subsection{Growth rate near of free boundary}\label{SecImpReg}

Before starting the proof of Theorem \ref{ThmGR1} let us introduce a class of functions and sets which play an essential role in our strategy. Our approach is inspired in \cite{KKPS} and \cite{LeeShah} (cf. \cite{LRS}, \cite{OS}, \cite{OSS} and \cite{SS} for a similar strategy in dead core settings).

\begin{definition}\label{Smartcalss} We say that $u \in \mathfrak{J}_p(B_1)$ if
\begin{enumerate}
  \item[\checkmark] $\left\|\div(|\nabla u|^{p-2}\nabla u)\right\|_{L^{\infty}(B_1)} = \left\|\lambda_0u_{+}^q\right\|_{L^{\infty}(B_1)} \leq 1$.
  \item[\checkmark] $0 \leq u \leq 1$.
  \item[\checkmark] $u(0)=0$.
\end{enumerate}
\end{definition}

Next, we define for $u \in \mathfrak{J}_p(B_1)$ the following set with a ``doubling type property''
\begin{equation}\label{EqDoubProp}
    \mathfrak{D}[u] \defeq \left\{j \in \mathbb{N}\cup\{0\}\colon  2^{\frac{p}{p-1-q}}\max\left\{1, \frac{1}{\mathfrak{C}_0}\right\} \mathcal{S}_{\frac{1}{2^{j+1}}}[u] \geq \mathcal{S}_{\frac{1}{2^j}}[u]\right\},
\end{equation}
where in this context $\mathfrak{C}>0$ is the universal constant from the non-degeneracy property given in Theorem \ref{LGR}. Moreover, observe that $\mathfrak{D}[u]$ is not empty since $j=0\in \mathfrak{D}[u]$ due to Theorem \ref{LGR}:
$$
   \mathcal{S}_{\frac{1}{2}}[u] \geq \mathfrak{C}_0\left(\frac{1}{2}\right)^{\frac{p}{p-1-q}} \geq \mathfrak{C}_0\left(\frac{1}{2}\right)^{\frac{p}{p-1-q}}\mathcal{S}_{1}[u] \quad \implies  \quad
  \mathcal{S}_{1}[u]\leq 2^{\frac{p}{p-1-q}}\max\left\{1,\frac{1}{\mathfrak{C}_0} \right\}\mathcal{S}_{\frac{1}{2}}[u].
$$

The following results are the key in order to prove Theorem \ref{ThmGR1}. Heuristically, they state that as the \textit{Thiele modulus} becomes flat enough, the limiting configuration of the $p-$dead-core problem imposes, according to Harnack inequality, that any solution must be identically zero.
Thus, by compactness methods, solutions to \eqref{Eqp-Lapla} become very flat with a sharp geometric decay (near free boundary points) provided $\lambda_0$ is under control.

\begin{lemma}\label{LemmaIter} There exists a positive constant $\mathfrak{C}_1 =  \mathfrak{C}_1(N, p, q, \|\lambda_0\|_{L^{\infty}(\Omega)})$ such that
  \begin{equation}\label{Eqiter}
     \mathcal{S}_{\frac{1}{2^{j+1}}}[u] \leq \mathfrak{C}_1 \left(\frac{1}{2^j}\right)^{\frac{p}{p-1-q}}
  \end{equation}
  for all $u \in \mathfrak{J}_p(B_1)$ and $j \in \mathfrak{D}[u]$.
\end{lemma}

\begin{proof} We proceed by contradiction. Suppose that the thesis of the lemma fails to hold. Then for each $k \in \mathbb{N}$ we may find $u_k \in \mathfrak{J}_p(B_1)$ and $j_k \in \mathfrak{D}[u]$ such that
  \begin{equation}\label{Eqcont}
     \mathcal{S}_{\frac{1}{2^{j_k+1}}}[u_k] > k \left(\frac{1}{2^{j_k}}\right)^{\frac{p}{p-1-q}}.
  \end{equation}
Now, define the auxiliary function
$$
   v_k(x) \defeq \frac{u_k(\frac{1}{2^{j_k}}x)}{\mathcal{S}_{\frac{1}{2^{j_k+1}}}[u_k]} \quad \mbox{in} \quad B_1,
$$
and notice that $v_k$ fulfills
\begin{enumerate}
  \item[\checkmark] $\displaystyle 0 \leq v_k(x) \leq \frac{\mathcal{S}_{\frac{1}{2^{j_k}}}[u_k]}{\mathcal{S}_{\frac{1}{2^{j_k+1}}}[u_k]} \leq 2^{\frac{p}{p-1-q}}\max\left\{1, \frac{1}{\mathfrak{C}_0}\right\}$ form \eqref{EqDoubProp} and $v_k(0) = 0$ for all $k$;
  \item[\checkmark] $\mathcal{S}_{\frac{1}{2}}[v_k] = 1$ for all $k$;
  \item[\checkmark] $\displaystyle -\div(|\nabla v_k|^{p-2}\nabla v_k) + \frac{1}{2^{j_k.p}}\frac{1}{\mathcal{S}^{p-1-q}_{\frac{1}{2^{j_k+1}}}[u_k]}\lambda_0\left(\frac{1}{2^{j_k}}x\right)(v_k)^q_{+}= 0 $.
\end{enumerate}
Moreover, by the contradiction assumption we get
$$
   \left\|\frac{1}{2^{j_k.p}}\frac{1}{\mathcal{S}^{p-1-q}_{\frac{1}{2^{j_k+1}}}[u_k]}\lambda_0\left(\frac{1}{2^{j_k}}x\right)(v_k)^q_{+}(x)\right\|_{L^{\infty}(B_1)} \leq \left(\frac{2^p}{k}\right)^{\frac{1}{p-1-q}}\max\left\{1, \frac{1}{\mathfrak{C}_0}\right\}\|\lambda_0\|_{L^{\infty}(B_1)}.
$$
Invoking the Harnack's inequality given in Theorem \ref{harnack} we obtain
$$
  1 = \mathcal{S}_{\frac{1}{2}}[v_k] \leq C(N, p)\left[\mathcal{I}_{\frac{1}{2}}[v_k] + \sqrt[p-1]{\left(\frac{2^{p}}{k}\right)^{\frac{1}{p-1-q}}\max\left\{1, \frac{1}{\mathfrak{C}_0}\right\}\|\lambda_0\|_{L^{\infty}(B_1)}}\ \right],
$$
which clearly gives a contradiction as $k \to \infty$. This concludes the proof.
\end{proof}

\begin{theorem}\label{THMAux} There exists a positive constant $\mathfrak{C}_2 =  \mathfrak{C}_2(N, p, q, \|\lambda_0\|_{L^{\infty}(\Omega)})$ such that for all $u \in \mathfrak{J}_p(B_1)$ it holds
$$
  u(x) \leq \mathfrak{C}_2|x|^{\frac{p}{p-1-q}}.
$$
\end{theorem}

\begin{proof} First, we claim that
\begin{equation}\label{EqFinalEst}
  \mathcal{S}_{\frac{1}{2^j}}[u] \leq \mathfrak{C}_1\left(\frac{1}{2^{j-1}}\right)^{\frac{p}{p-1-q}} \quad \forall \,\, j \in \mathbb{N},
\end{equation}
where $\mathfrak{C}_1$ is the constant from Lemma \ref{LemmaIter}. We will prove this claim by induction.
Without loss of generality we can assume that $\mathfrak{C}_2 \geq 1$, and then \eqref{EqFinalEst} holds for $j=0$.
Suppose that \eqref{EqFinalEst} holds for some $j \in \mathbb{N}$. We will verify the $(j+1)^{\text{th}}$ step. In fact, if $ j \in \mathfrak{D}[u]$ then the result holds by Lemma \ref{LemmaIter}. On the other hand, if $j \notin \mathfrak{D}[u]$, from the inductive hypothesis we get
$$
   \mathcal{S}_{\frac{1}{2^{j+1}}}[u] \leq \left(\frac{1}{2}\right)^{\frac{p}{p-1-q}}\mathcal{S}_{\frac{1}{2^j}}[u] \leq  \left(\frac{1}{2}\right)^{\frac{p}{p-1-q}} \mathfrak{C}_1 \left(\frac{1}{2^{j-1}}\right)^{\frac{p}{p-1-q}} = \mathfrak{C}_1 \left(\frac{1}{2^{j}}\right)^{\frac{p}{p-1-q}},
$$
and then \eqref{EqFinalEst} holds for all $j \in \mathbb{N}$.
Finally,  given $r \in (0, 1)$ let $j \in \mathbb{N}$ be the greatest integer such that $\frac{1}{2^{j+1}} \leq r < \frac{1}{2^j}$. Then,
\begin{equation}\label{Cont_est}
 \mathcal{S}_{r}[u] \leq \mathcal{S}_{\frac{1}{2^{j}}}[u] \leq \mathfrak{C}_1
 \left(\frac{1}{2^{j-1}}\right)^{\frac{p}{p-1-q}}
 \leq \mathfrak{C}_2(N, p, q, \|\lambda_0\|_{L^{\infty}(\Omega)})r^{\frac{p}{p-1-q}}
\end{equation}
and the proof is finished.
\end{proof}

Next, we  present a universal normalization and scaling procedure in order to place solutions of \eqref{Eqp-Lapla} in an appropriated flatness scenery, which allows us to apply Lemma \ref{LemmaIter}. This reasoning is inspired in the  flatness devise introduced in \cite{Tei16} and \cite{Tei-16}.

\begin{remark}[{\bf Flatness scenery}]\label{FlatHip} Notice that Lemma \ref{LemmaIter} assures the existence of a universal constant $0<\tau_0 \ll 1$ (small enough) such that if $u \in \mathfrak{J}_p(B_1)$ with
$$
   \left\|\div(|\nabla u|^{p-2} \nabla u)\right\|_{L^{\infty}(B_1)}  = \left\|\lambda_0 u_{+}^q\right\|_{L^{\infty}(B_1)}\leq \tau_0,
$$
then
$$
   \mathcal{S}_{\frac{1}{2^{j+1}}}[u] \leq \mathfrak{C}_{\tau_0}(p, q, N)
   \left(\frac{1}{2^j}\right)^{\frac{p}{p-1-q}}.
$$
Therefore, we are able to select $\tau_0\ll 1$ such that
$$
  \mathfrak{C}_{\tau_0}(p, q, N) =  \frac{2}{2^{\frac{p}{p-1-q}}}  \mathfrak{C} _0.
$$
Particularly, coming back to estimate \eqref{Cont_est}, we obtain
\begin{equation}\label{Sharp_est}
 \mathcal{S}_{r}[u] \leq 2 \cdot 2^{\frac{p}{p-1-q}} \mathfrak{C}_0  r^{\frac{p}{p-1-q}}.
\end{equation}
\end{remark}

We now are ready to prove Theorem \ref{ThmGR1}.

\begin{proof}[{\bf Proof of  Theorem \ref{ThmGR1}}]\label{rem1} In order to prove  Theorem \ref{ThmGR1}, we have to reduce its hypotheses to the framework of Theorem \ref{THMAux}. We assume without loss of generality that  $\overline{B_1} \Subset \Omega$. For $x_0 \in \partial \{u>0\} \cap \overline {B_1}$ we will proceed with a normalization and intrinsic scaling argument: let us define
$$
     v(x)\defeq \frac{u\left( x_0 + \mathfrak{R}_0 x\right)}{\kappa_0} \quad \mbox{in} \quad B_1
$$
for constants $\kappa_0, \mathfrak{R}_0>0$ to be determined universally \textit{a posteriori}.

From the equation satisfied by $u$, we easily verify that $v$ fulfills
\begin{equation}\label{eqcomp}
    -\Delta_p v + \hat{\lambda}_0(x) v^q_{+}(x)= 0,
\end{equation}
in the weak sense for $\hat{\lambda}_0(x)\defeq  \frac{\mathfrak{R}_0^{p}}{\kappa_0^{p-1-q}}\lambda_0(x_0+\mathfrak{R}_0x).$
Now, let $\tau_0>0$ be the greatest universal constant pointed out in Remark \ref{FlatHip} such that Lemma  \ref{LemmaIter} holds provided that
 $$
   \|\div(|\nabla u|^{p-2} \nabla u)\|_{L^{\infty}(B_1)} \leq \tau_0.
 $$
By choosing
\begin{equation}\label{EqNScond}
  \kappa_0 \defeq \|u\|_{L^{\infty}(\Omega)} \quad \text{and} \quad 0<\mathfrak{R}_0< \min\left\{1, \frac{\dist(\overline{B_1}, \partial \Omega)}{2}, \sqrt[p]{\frac{\tau_0\kappa_0^{p-1-q}}{\|\lambda_0\|_{L^{\infty}(\Omega)}}} \right\},
\end{equation}
then $v$ fits into the framework of Theorem \ref{THMAux}. Hence, there exists $ \mathfrak{C}_2= \mathfrak{C}_2(N, p, q, \inf_{\Omega} \lambda_0(x))$ such that
$$
  v(x) \leq \mathfrak{C}_2|x|^{\frac{p}{p-1-q}}.
$$
By scaling back, we obtain the conclusion of Theorem \ref{ThmGR1}.
\end{proof}

 As a consequence of Theorem \ref{ThmGR1} we obtain a finer decay near free boundary points. Precisely, a dead core solution $u$ arrives at its null set as a suitable power of the distance up to the free boundary.

\begin{corollary}\label{DistEst} Let $u$ be a bounded weak solution to \eqref{Eqp-Lapla} and $\Omega^{\prime} \Subset \Omega$. Then, there is a constant $\mathfrak{C}^{\sharp}=\mathfrak{C}^{\sharp}\left(N, p, q, \inf_{\Omega} \lambda_0(x)\right)$ such that, for any point $x_0 \in   \{u > 0\} \cap \Omega^{\prime}$ such that $\dist(x_0, \partial \{u>0\}) \leq \frac{\mathfrak{R}_0}{2}$, there holds
\begin{equation}
 \displaystyle  u(x_0) \leq \mathfrak{C}^{\sharp}
[ \dist(x_0, \partial \{u>0\})]^{\frac{p}{p-1-q}}.
\end{equation}
\end{corollary}
\begin{proof}
Fix $x_0 \in   \{u > 0\} \cap \Omega^{\prime}$ and denote $d\defeq  \dist(x_0, \partial \{u>0\})$. Now, select $z_0 \in \partial \{u>0\}$ a free boundary point which achieves the distance, i.e., $d = |x_0-z_0|$. From  Theorem \ref{ThmGR1} we have that
\begin{equation}\label{eq4.4}
  \displaystyle u(x_0) \leq \sup_{B_{d}(x_0)} u(x) \leq \sup_{B_{2d}(z_0)} u(x) \leq \mathfrak{C}^{\sharp}\left(N, p, q, \inf_{\Omega} \lambda_0(x)\right) d^{\frac{p}{p-1-q}}.
\end{equation}
This finishes the proof.
\end{proof}

By using an argument based on Serrin's Harnack inequality, the first and third author established in \cite[Theorem 1.1]{SS}, for a more general class of $p-$Laplacian type operators, an alternative version for the growth estimates close to the free boundary points which reads as
\begin{equation}\label{estim.77}
  \displaystyle \sup_{B_r(x_0)} u(x) \leq \mathfrak{C}_{\sharp} \max\left\{\inf_{B_r(x_0)} u(x), \,\, r^{\frac{p}{p-1-q}} \right\},
\end{equation}
for a universal constant $\mathfrak{C}_{\sharp}>0$ and all $0< r< \min\left\{1, \frac{\dist(x_0, \partial \Omega)}{2}\right\}$.

Here we have decided to adopt the iterative geometric decay, namely Lemma \ref{LemmaIter}, due to its sharp and explicit representation for the universal constants involved in its estimate.

\subsection{Applications}\label{SecApplic}

In this section we present a number of interesting applications of our previous results.
Our first application establishes a similar growth rate for the gradient of functions $u \in \mathfrak{J}_p(B_1)$.

\begin{lemma}\label{IRresult2.66} There exists a constant $\mathfrak{C}_3 =  \mathfrak{C}_3(N, p, q, \lambda_0)>0$ such that for all $u \in \mathfrak{J}_p(B_1)$ there holds
$$
  |\nabla u(x)| \leq \mathfrak{C}_3|x|^{\frac{1+q}{p-1-q}} \quad \forall \,\, x \in B_{\frac{1}{2}}.
$$
\end{lemma}

\begin{proof} As previously, it is enough to prove the following estimate
\begin{equation}\label{eqEstGrad}
  \mathcal{S}_{\frac{1}{2^{j+1}}}[| \nabla u|] \leq \max\left\{\mathfrak{C} \left(\frac{1}{2^j}\right)^{\frac{1+q}{p-1-q}}, \left(\frac{1}{2}\right)^{\frac{1+q}{p-1-q}}\mathcal{S}_{\frac{1}{2^{j}}}[|\nabla u|]\right\},
\end{equation}
for all $j \in \mathbb{N}$ and a constant $\mathfrak{C} = \mathfrak{C}(N, p, q, \lambda_0)$.
Suppose, arguing by contradiction, that \eqref{eqEstGrad} is not true. Then, there exists $u_j \in \mathfrak{J}_p(B_1)$ such that
\begin{equation}\label{eqEstGradHip}
  \mathcal{S}_{\frac{1}{2^{j+1}}}[|\nabla u_j|] \geq \max\left\{j\left(\frac{1}{2^j}\right)^{\frac{1+q}{p-1-q}}, \left(\frac{1}{2}\right)^{\frac{1+q}{p-1-q}}\mathcal{S}_{\frac{1}{2^{j}}}[|\nabla u_j|]\right\}.
\end{equation}
Now, we define the auxiliary function
$$
   v_j(x) \defeq \frac{2^ju_j(\frac{1}{2^j})}{\mathcal{S}_{\frac{1}{2^{j+1}}}[|\nabla u_j|]} \quad \mbox{for} \quad x \in B_1.
$$
Hence, by using \eqref{Eqiter} and \eqref{eqEstGradHip} we get
$$
  0\leq  v_j(x) \leq \frac{2^j\mathfrak{C}_1(2^{-j})^{\frac{p}{p-1-q}}}{\mathcal{S}_{\frac{1}{2^{j+1}}}[|\nabla u_j|]}\leq \frac{\mathfrak{C}_1}{j} \quad \mbox{for} \quad x \in B_1.
$$
Furthermore, $\mathcal{S}_{\frac{1}{2}}[|\nabla v_j|]  = 1$ and $\mathcal{S}_{1}[|\nabla v_j|]  \leq 2^{\frac{1+q}{p-1-q}}$. We also have that
$$
     -\Delta_p v_j + \hat{\lambda^j_0}(x)(v_j)^{q}_{+}(x)  = 0\quad \mbox{in} \quad B_1
$$
in the weak sense, where $$\hat{\lambda^j_0}(x) \defeq \frac{1}{2^{j(1+q)}}\frac{1}{\mathcal{S}^{p-1-q}_{\frac{1}{2^{j+1}}}[|\nabla u_j|]}\lambda^j_0\left(\frac{1}{2^{j}}x\right).$$ Hence, we get that
$$
   \left\|\hat{\lambda^j_0}(v_j)^{q}_{+}\right\|_{L^{\infty}(B_1)} \leq \mathfrak{C}^{q}
   \sup_{B_1} \lambda_0(x) \left(\frac{1}{j}\right)^{p-1}.
$$
Finally, by invoking the uniform gradient estimates from \cite{Choe1}, \cite{DB0} and \cite{Tolk} we obtain
$$
  1= \mathcal{S}_{\frac{1}{2}}[| \nabla v_j|] \leq \mathfrak{C}(N, p)\left[\|v_j\|_{L^{\infty}(B_1)}+ \left\|\hat{\lambda^j_0}(v_j)^{q}_{+}\right\|^{\frac{1}{p-1}}_{L^{\infty}(B_1)}\right] \leq \mathfrak{C}^{\ast}\left(\frac{1}{j}+\frac{1}{j}\right) \to 0 \quad \mbox{as} \quad j \to \infty,
$$
which clearly yields a contradiction and the lemma is proved.
\end{proof}

\begin{remark} Similarly to Remark \ref{FlatHip}, we obtain for any $x_0 \in \partial \{u>0\} \cap \Omega^{\prime}$ and $r\ll 1$ the following
\begin{equation}\label{EquEstGrad}
    \displaystyle \sup_{B_r(x_0)} |\nabla u(x)| \leq \mathfrak{C}(N, p, q, \lambda_0) r^{\frac{1+q}{p-1-q}},
\end{equation}
where in this case we are able to choose $\mathfrak{C}(N, p, q, \lambda_0) = 2^{\frac{1+q}{p-1-q}}\mathfrak{C}_0$.
\end{remark}

We stress that with a similar strategy one can estimate the $L^2-$norm of $|\nabla u(x)|^{p-2}|D^2 u(x)|$, which may not exist  point-wisely.

\begin{lemma}[{\bf Estimate in $L^2-$average}]\label{LemmL2Est}
 For every $u \in \mathfrak{J}_p(B_1)$ with $p>2$ and $x_0 \in \partial \{u>0\} \cap B_{\frac{1}{2}}$ there exists $M = M(N, p, q, \lambda_0)$ such that
$$
   \displaystyle \left(\intav{B_r(x_0)} (|\nabla u(x)|^{p-2}|D^2 u(x)|)^2 dx\right)^{\frac{1}{2}} \leq M
   r^{\frac{pq}{p-1-q}}.
$$
\end{lemma}
\begin{proof}
Note that by translation and scaling we might assume that $x_0 = 0$. Now, let us define
$$
  \hat{\mathcal{S}}_r[u] \defeq \displaystyle \left(\intav{B_1} (|\nabla u(rx)|^{p-2}|D^2 u(rx)|)^2 dx\right)^{\frac{1}{2(p-1)}}.
$$
As before, it suffices to prove that there exists a universal constant $\mathfrak{C}_{\ast}$ such that
\begin{equation}\label{EqEstL2}
  \hat{\mathcal{S}}_{\frac{1}{2^{k+1}}}[u] \leq \max\left\{\mathfrak{C}_{\ast}\left(\frac{1}{2^k}\right)^{\frac{pq}{(p-1)(p-1-q)}}, \left(\frac{1}{2}\right)^{\frac{pq}{(p-1)(p-1-q)}}\hat{\mathcal{S}}_{\frac{1}{2^{k}}}[u]\right\} \quad \forall \,\, k \in \mathbb{N} \ \text{and} \ 0<r\leq \frac{1}{2}.
\end{equation}
As before, we proceed by contradiction and  assume that there exist $u_k \in \mathfrak{J}_p(B_1)$ and $j_k \in \mathbb{N}$ such that \eqref{EqEstL2} does not hold. Therefore,
$$
\begin{array}{rcl}
  \displaystyle \hat{\mathcal{S}}_{\frac{1}{2^{j_k+1}}}[u_k]  & \geq &  \displaystyle \max\left\{k\left(\frac{1}{2^{j_k}}\right)^{\frac{pq}{(p-1)(p-1-q)}}, \left(\frac{1}{2}\right)^{\frac{pq}{(p-1)(p-1-q)}}\hat{\mathcal{S}}_{\frac{1}{2^{j_k}}}[u_k]\right\} \\[10pt]
   & \geq & \displaystyle \max\left\{k\left(\frac{1}{2^{j_k}}\right)^{\frac{p}{p-1-q}}, \left(\frac{1}{2}\right)^{\frac{pq}{(p-1)(p-1-q)}}\hat{\mathcal{S}}_{\frac{1}{2^{j_k}}}[u_k]\right\}.
\end{array}
$$
Now, let us define the scaled and ``normalized'' function $v_k: B_1 \to \R$ by
$$
  v_k(x) \defeq \frac{u_k(\frac{1}{2^{j_k}}x)}{\hat{\mathcal{S}}_{\frac{1}{2^{j_k+1}}}[u_k]}.
$$
Thus, it is easy to verify that
\begin{itemize}
  \item[\checkmark] $\|v_k\|_{L^{\infty}(B_1)} \leq \frac{\mathfrak{C}}{k}$;
  \item[\checkmark] $\hat{\mathcal{S}}_{\frac{1}{2}}[v_k] = 1$;
  \item[\checkmark] $\displaystyle \hat{\mathcal{S}}_{1}[v_k] \leq 2^{\frac{pq}{(p-1)(p-1-q)}}$;
  \item[\checkmark] $\displaystyle -\Delta_p v_k(x) + \frac{1}{2^{j_k.p}}\frac{1}{\hat{\mathcal{S}}^{p-1-q}_{\frac{1}{2^{j_k+1}}}[u_k]}\lambda_0\left(\frac{1}{2^{j_k}}x\right)v_k^q(x) = 0$ in the weak sense in $B_1$.
\end{itemize}
Moreover, we have that
$$
   \displaystyle \left\|\lambda_0(v_k)^{q}_{+}\right\|_{L^{\infty}(B_1)} \leq \mathfrak{C}^{q}
   \sup_{B_1} \lambda_0(x) \left(\frac{1}{k}\right)^{p-1}.
$$
On the other hand, by invoking the uniform gradient estimates from \cite{Choe1}, \cite{DB0} and \cite{Tolk} and the $L^2$-bound of the second derivatives from \cite{Tolk} we obtain
$$
  \mathcal{S}_{\frac{1}{2}}[|\nabla v_k|] \leq \mathfrak{C}(N, p)\left[\|v_k\|_{L^{\infty}(B_1)}+ \left\|\lambda_0(v_k)^{q}_{+}\right\|^{\frac{1}{p-1}}_{L^{\infty}(B_1)}\right] \leq \frac{\mathfrak{C}^{\ast}}{k}
\; \text{  and  } \;
  \left(\intav{B_1} |D^2 v_k(x)|^2 dx\right)^{\frac{1}{2}} < C< \infty.
$$
Finally,
$1=\hat{\mathcal{S}}_{\frac{1}{2}}[v_k]<C^{\frac{1}{p-1}}k^{-\frac{p-2}{p-1}}$, which yields a contradiction for $k $ large when  $p>2$.
\end{proof}

The following result gives a non-degeneracy estimate for the Hessian at free boundary points in the $L^2-$average sense.

\begin{lemma}[{\bf Non-degeneracy for the Hessian in $L^2-$average}]
Let $u$ be a bounded weak solution to \eqref{Eqp-Lapla}, $q\geq \max\left\{\frac{1}{2},\frac{p-2}{2}\right\}$ and $\Omega^{\prime} \Subset \Omega$. If $u$ is strong non-degenerate in $L^2-$average, i.e.,
$$
\intav{B_r(x_0)} u(x)dx \geq C(N, p, q)r^{\frac{p}{p-1-q}}
$$
for any point $x_0 \in \partial \{u>0\} \cap \Omega^{\prime}$, then, there holds that
$$
   \displaystyle \left(\intav{B_r(x_0)} |D^2 u(x)|^2 dx\right)^{\frac{1}{2}} \geq C(N, p, q, \lambda_0)r^{\frac{2(q+1)-p}{p-1-q}}.
$$
\end{lemma}
\begin{proof} Recall that the $p-$Laplace operator can be decomposed in the non-divergence form
$$
  \Delta_p u(x) = |\nabla u(x)|^{p-2}\Delta u(x) + (p-2)|\nabla u(x)|^{p-4}\Delta_{\infty} u(x).
$$
Moreover, it is easy to check that
\begin{align*}
  \displaystyle \left(\Delta_\infty u(x)\right)^2  =   \displaystyle \left(\sum_{i, j=1}^{N}  u_j(x) u_{ij}(x) u_i(x)\right)^2
  &\leq    \displaystyle \left(\sum_{i, j=1}^{N} u^2_{ij}(x) \right)\left(\sum_{i, j=1}^{N}  u^2_j(x) u^2_i(x)\right) \\
   & =  |D^2 u(x)|^2|\nabla u(x)|^4.
\end{align*}
Hence,
\begin{equation}\label{Eq5.4}
\begin{array}{l}
  \displaystyle \intav{B_r(x_0)} (\lambda_0(x)u_{+}^q(x))^2dx  =  \displaystyle \intav{B_r(x_0)} (\Delta_p u(x))^2 dx \\
   \qquad =   \displaystyle \intav{B_r(x_0)} \left[|\nabla u(x)|^{p-2}\Delta u(x) + (p-2)|\nabla u(x)|^{p-4}\Delta_{\infty} u(x)\right]^2 dx\\
   \qquad \leq  \displaystyle \intav{B_r(x_0)} \left[|\nabla u(x)|^{p-2}|\Delta u(x)| + (p-2)|\nabla u(x)|^{p-4}|\Delta_{\infty} u(x)|\right]^2 dx\\
   \qquad \leq  \displaystyle (p-1)^2 \intav{B_r(x_0)} \left[|\nabla u(x)|^{p-2}|D^2 u(x)|\right]^2 dx.
\end{array}
\end{equation}
According to the $L^2-$bounds in \cite{Tolk}, the last integral is finite. Now, using the Strong Non-degeneracy in $L^2-$average and Jensen's inequality we obtain
\begin{equation}\label{Eq5.5}
    \displaystyle \intav{B_r(x_0)} (\lambda_0(x)u_{+}^q(x))^2dx  \geq  \left(\inf_{\Omega} \lambda_0(y) C(N, p, q) r^{\frac{pq}{p-1-q}}\right)^2.
\end{equation}
On the other hand, from Lemma \ref{IRresult2.66}
\begin{equation}\label{Eq5.6}
    \displaystyle |\nabla u(x)| \leq \sup_{B_r(x_0)} |\nabla u(x)| \leq \mathfrak{C}_3(N, p, q, \lambda_0)
    r^{\frac{1+q}{p-1-q}}.
\end{equation}
Putting \eqref{Eq5.5} and \eqref{Eq5.6} into \eqref{Eq5.4} we conclude that
$$
\intav{B_r(x_0)} |D^2 u(x)|^2 dx \geq \left(\frac{\inf\limits_{\Omega} \lambda_0(y) C(N, p, q)}{(p-1)\mathfrak{C}(N, p, q, \lambda_0)^{p-2}}\right)^2 r^{\frac{2[2(1+q)-p]}{p-1-q}},
$$
which finishes the proof of the Lemma.
\end{proof}

We end this section with an application that can be used for the numeric approximation of the free boundary, a stability result. To this end, decompose $\Omega$ into finite elements
and let $h \in \R_{+}$ be a discretization parameter that converges to zero, for example, $h$ might refer to the mesh size. Now let $u_{h}$ be the corresponding discrete solution to \eqref{Eqp-Lapla}. In contrast with other free boundary problems, it is possible that $\partial \{u_{h}>0\}\cap \Omega = \emptyset$. For this reason, it is convenient to define the discrete free boundary by
$$
  \mathfrak{F}^h_{\Omega}[u_h] \defeq \partial \{u_h>\delta_h\}\cap \Omega,
$$
where $\delta_h>0$ is a parameter to be determined \textit{a posteriori}. Finally, let us assume an $L^s-$error estimate for the solutions: we assume that there exists a continuous modulus of continuity $\mu: [0, \infty) \to [0, \infty)$ with $\mu(0)=0$ such that, for some $s \in [1, \infty]$ there holds
\begin{equation}\label{EqErrAppr}
  \|u-u_h\|_{L^s(\Omega)} <\mu(h).
\end{equation}
Under these assumptions we establish the following \textit{error estimate in measure} for the free boundary.
\begin{lemma} Under the same assumptions of Theorem \ref{ThmNDinM}, take $\delta_h \defeq \mu(h)^{\frac{s(p-1-q)}{p-1-q+sp}}$, then, there exists a constant $\mathfrak{B}$ such that
$$
   \Leb\left(\mathfrak{F}_{\Omega}[u] \bigtriangleup \mathfrak{F}^h_{\Omega}[u_h]\right) \leq \mathfrak{B}\, \delta_h^{\frac{p-1-q}{p}}.
$$
Additionally, if the error estimate \eqref{EqErrAppr} holds for $s=\infty$, and there exists  $\mathfrak{C}(p, q, N, \Omega^{\prime})>0$ such that
$$
  \Omega^{\prime} \cap \left\{0<u<\varepsilon^{\frac{p}{p-1-q}}\right\} \subset \mathcal{N}_{\varepsilon}\left(\mathfrak{F}_{\Omega}[u]\cap \Omega^{\prime}\right),
$$
where $\mathcal{N}_{\varepsilon}(\mathfrak{S}) \defeq \{x \in \Omega: \dist(x, \mathfrak{S}) \leq \mathfrak{C}(p, q, N, \Omega^{\prime}) \varepsilon\}$,
then
$$
   \mathfrak{F}^h_{\Omega}[u_h]\cap \Omega^{\prime} \subset \mathcal{N}_{(2\mu(h))^{\frac{p-1-q}{p}}}(\mathfrak{F}_{\Omega}[u]\cap \Omega^{\prime}).
$$
\end{lemma}
\begin{proof}
The proof holds directly as consequence of \cite[Theorem 2.1 and Theorem 2.2]{Noch} once that the hypotheses of such theorems are checked from Theorems \ref{ThmGR1}, \ref{ThmNDinM} and assumption \eqref{EqErrAppr}.
\end{proof}

\section{The limit problem}

This section is devoted to prove Theorems \ref{MThmLim1}, \ref{MThmLim3} and \ref{MThmLim2} concerning the limit as $p\to \infty$. First, we will prove the existence of a uniform limit for Theorem \ref{MThmLim1} as $p\to\infty$. Notice that since the boundary datum $g$ is assumed to be Lipschitz continuous we can extend it to a Lipschitz function
(that we will still call $g$) to the whole $\Omega$.

\begin{lemma}\label{Lemma2.4} Assume $\max\{2, N\}<p < \infty$ and let $u \in W^{1, p}(\Omega)$ be a weak solution to \eqref{Eqp-Lapla}. Then,
$$
\|\nabla u\|_{L^p(\Omega)} \leq C_1.
$$
Additionally, $u \in C^{0, \alpha}(\Omega)$, where $\alpha = 1- \frac{N}{p}$ with the following estimate
  $$
  \frac{|u(x)-u(y)|}{|x-y|^{\alpha}} \leq C_2.
  $$
Here $C_1, C_2>0$ are constants depending on $N$, $p$, $q$, $ \|\lambda_0\|_{L^{\infty}(\Omega)}$, $\|u\|_{L^{p}(\Omega)}$, $\|g\|_{L^{p}(\Omega)}$, $\|\nabla g\|_{L^{p}(\Omega)}$.
\end{lemma}

\begin{proof}
By multiplying \eqref{Eqp-Lapla} by $u$ and using integration by parts and H\"{o}lder inequality we obtain
$$
\begin{array}{rcl}
  \displaystyle \int_{\Omega} |\nabla u|^p \, dx & = & \displaystyle  \int_{\partial \Omega} g|\nabla g|^{p-2}\nabla g\, dx - \int_{\Omega} \lambda_0(x)u_{+}^{q+1}(x)dx \\
   & \leq  & \displaystyle \Leb(\Omega)^{1-\frac{q+1}{p}}\|\lambda_0\|_{L^{\infty}(\Omega)}\|u\|^{q+1}_{L^{p}(\Omega)} + \|g\|_{L^p(\partial \Omega)}\|\nabla g\|^{p-1}_{L^p(\partial \Omega)}.
\end{array}
$$
Therefore,
$$
\|\nabla u\|_{L^p(\Omega)} \leq \left(\Leb(\Omega)^{1-\frac{q+1}{p}}\|\lambda_0\|_{L^{\infty}(\Omega)}\|u\|^{q+1}_{L^{p}(\Omega)} + \|g\|_{L^p(\partial \Omega)}\|\nabla g\|^{p-1}_{L^p(\partial \Omega)}\right)^{\frac{1}{p}}.
$$
Next, for $p>N$ by Morrey's estimates and the previous sentence, there exists a positive constant $\mathfrak{C}=\mathfrak{C}(N, p, \Omega)$ such that
$$
   \frac{|u(x)-u(y)|}{|x-y|^{1- \frac{N}{p}}} \leq \mathfrak{C} \|\nabla u\|_{L^p(\Omega)}.
$$
\end{proof}

The next result assures that any family of weak solutions to \eqref{Eqp-Lapla} is pre-compact and therefore the existence of a uniform limit for Theorem \ref{MThmLim1} is guaranteed.

\begin{lemma}[{\bf Existence of limit solutions}]\label{LemExistSol} Let $\{u_p\}_{p>1}$ be a sequence of weak solutions to \eqref{Eqp-Lapla}. Then, there exists a subsequence $p_j \to \infty$ and a limit function $u_{\infty}$ such that
$$
   \displaystyle \lim_{p_j \to \infty} u_{p_j}(x) = u_{\infty}(x)
$$
uniformly in $\Omega$. Moreover, $u_{\infty}$ is Lipschitz continuous with
$$
    [u_{\infty}]_{\text{Lip}(\overline{\Omega})} \leq \limsup_{p_j \to \infty} \mathfrak{C}(N, p_j, \Omega)\|\nabla u_{p_j}\|_{L^{p_j}(\Omega)} \leq \mathfrak{C}(N)\max\left\{\|u\|^{\ell}_{L^{\infty}(\Omega)}, [g]_{\text{Lip}(\partial \Omega)}\right\}.
$$
\end{lemma}
\begin{proof} Due to our previous uniform bound for the gradient of $u_p$ and the fact that the boundary datum is fixed, we obtain that $\| u_{p}\|_{L^\infty (\Omega)} \leq \mathfrak{C}$ for all $2 \leq p< \infty$. Hence, existence of $u_{\infty}$ as an uniform limit is a direct consequence of the Lemma \ref{Lemma2.4} combined with the Arzel\`{a}-Ascoli compactness criterium. Finally, the last statement holds by passing to the limit in the H\"{o}lder's estimates from Lemma \ref{Lemma2.4}.
\end{proof}

Now, we show that any uniform limit, $u_\infty$, is a viscosity solution to the limit equation.
Recall that we assumed that $\displaystyle \ell \defeq \lim_{p \to \infty} q(p)/p \in [0, 1)$ exists.

\begin{proof}[{\bf Proof of Theorem \ref{MThmLim1}}]
First, observe that from the uniform convergence, it holds that $u_{\infty} = g$ on $\partial \Omega$. Next, we   prove that the limit function $u_{\infty}$ is an $\infty$-harmonic function in its null set, this is,
$$
  - \Delta_{\infty} u_{\infty}(x) = 0 \quad \text{in} \quad \{u_{\infty} = 0\} \cap \Omega.
$$
To this end, let $x_0 \in \{u_{\infty} = 0\} \cap \Omega$ and $\phi \in C^2(\Omega)$ such that $u_{\infty}-\phi$ has a strict local maximum (resp. strict local minimum) at $x_0$. Since, up to a subsequence,  $u_p \to u_{\infty}$ locally uniformly, there exists a sequence $x_p \to x_0$ such that $u_p-\phi$ has a local maximum (resp. local minimum) at $x_p$. Moreover, if $u_p$ is a weak solution to \eqref{Eqp-Lapla} (consequently a viscosity solution, by Lemma \ref{EquivSols})  we obtain
$$
  -\left[|\nabla \phi(x_p)|^{p-2}\Delta \phi(x_p) + (p-2)|\nabla \phi(x_p)|^{p-4}\Delta_{\infty} \phi(x_p)\right] \leq -\lambda_0(x_p)\phi_{+}^q(x_p) \quad (\text{resp.}\,\, \geq ).
$$
Now, if $|\nabla \phi(x_0)| \neq 0$ we may divide both sides of the above inequality by $(p-2)|\nabla \phi(x_p)|^{p-4}$ (which is different from zero for $p$ large enough). Thus, we obtain that
$$
 - \Delta_{\infty} \phi(x_p) \leq \frac{|\nabla \phi(x_p)|^2 \Delta \phi(x_p)}{p-2} - \frac{ \lambda_0(x_p)\phi_{+}^q(x_p)}{(p-2)|\nabla \phi(x_p)|^{p-4}} \quad (\text{resp.}\,\, \geq ),
$$
where the RHS tends to zero as $p \to \infty$. Therefore,
$$
    - \Delta_{\infty} \phi(x_0) \leq 0 \quad (\text{resp.}\,\, \geq 0).
$$
Finally, since such an inequality is also satisfied if $|\nabla \phi(x_0)| = 0$ we conclude that $u_{\infty}$ is a viscosity sub-solution (resp. super-solution) in its null set.

Next, we will prove that $u_{\infty}$ is a viscosity solution to
$$
  \max\left\{-\Delta_{\infty} u_{\infty}(x), -|\nabla u_{\infty}(x)|+ u_{\infty}^{\ell}(x)\right\} = 0 \quad \text{in} \quad \{u_{\infty}>0\} \cap \Omega.
$$
First, let us prove that $u_{\infty}$ is a viscosity super-solution. To that end, fix $x_0 \in \{u_{\infty}>0\} \cap \Omega$ and let  $\phi \in C^2(\Omega)$ be a test function such that $u_{\infty}(x_0) = \phi(x_0)$ and the inequality $u_{\infty}(x) > \phi(x)$ holds for all $x \neq x_0$. We want to show that
$$
  - \Delta_{\infty} \phi(x_0) \geq 0 \quad \text{or} \quad -|\nabla \phi(x_0)|+ \phi^{\ell}(x_0) \geq 0.
$$
Notice that if $|\nabla \phi(x_0)| = 0$ there is nothing to prove. Hence, we may assume that
\begin{equation}\label{eq5.1}
  -|\nabla \phi(x_0)|+\phi^{\ell}(x_0)<0.
\end{equation}
As in the previous case, there exists a sequence $x_{p} \to x_0$ such that $u_{p}-\phi$ has a local minimum at $x_{p}$. Since $u_{p}$ is a weak super-solution (resp. a viscosity super-solution by Lemma \ref{EquivSols}) to \eqref{Eqp-Lapla} we get
$$
  -\left[|\nabla \phi(x_{p})|^{p-2}\Delta \phi(x_{p}) + (p-2)|\nabla \phi(x_{p})|^{p-4}\Delta_{\infty} \phi(x_{p})\right] \geq -\lambda_0(x_p)\phi_{+}^{q}(x_p).
$$
Now, dividing both sides by $(p-2)|\nabla \phi(x_{p})|^{p-4}$ (which is not  zero for $p\gg1$ due to \eqref{eq5.1}) we get
$$
  - \Delta_{\infty} \phi(x_{p}) \geq  \frac{|\nabla \phi(x_{p})|^2 \Delta \phi(x_{p})}{p-2} -
   \frac{\lambda_0(x_p)\phi_{+}^{q}(x_p)}{(p-2)|\nabla \phi(x_{p})|^{p-4}}.
$$
Passing to the limit as $p \to \infty$ in the above inequality we conclude that
$$
- \Delta_{\infty} \phi(x_0) \geq 0,
$$
which proves that $u_{\infty}$ is a viscosity super-solution.

The last part consists in proving that $u_{\infty}$ is a viscosity sub-solution. To this end, fix $x_0 \in \{u_{\infty}>0\} \cap \Omega$ and a test function $\phi \in C^2(\Omega)$ such that $u_{\infty}(x_0) = \phi(x_0)$ and the inequality $u_{\infty}(x) < \phi(x)$ holds for $x \neq x_0$. We want to prove that
\begin{equation}\label{eq5.2}
  - \Delta_{\infty} \phi(x_0) \leq 0 \quad \text{and} \quad -|\nabla \phi(x_0)|+\phi^{\ell}(x_0) \leq 0.
\end{equation}
One more time, there exists a sequence $x_{p} \to x_0$ such that $u_{p}-\phi$ has a local maximum at $x_{p}$ and since $u_{p}$ is a weak sub-solution (resp. viscosity sub-solution) to \eqref{Eqp-Lapla}, we have that
$$
 - \frac{|\nabla \phi(x_{p})|^2 \Delta \phi(x_{p})}{p-2} - \Delta_{\infty} \phi(x_{p}) \leq  -
   \frac{\lambda_0(x_p)\phi_{+}^{q}(x_p)}{(p-2)|\nabla \phi(x_{p})|^{p-4}} \leq 0.
$$
Thus, letting $p \to \infty$ we obtain $- \Delta_{\infty} \phi(x_0) \leq 0$. Moreover, if  $-|\nabla \phi(x_0)|+ \phi^{\ell}(x_0) > 0$, as $p \to \infty$, then the right hand side diverges to $-\infty$, causing a contradiction.
Therefore \eqref{eq5.2} holds.
\end{proof}

\begin{proof}[{\bf Proof of Theorem \ref{MThmLim3}}]
 From Lemma \ref{LemExistSol} any sequence of bounded weak solutions $(u_p)_{p>2}$ converges, up to a subsequence, to a limit, $u_{\infty}$, uniformly in $\Omega$. This limit $u_{\infty}$ also fulfills \eqref{EqLim} in the viscosity sense. On the other hand, \eqref{Sharp_est} claims that, for $0<r\ll 1$,
$$
 \displaystyle \sup_{B_r(x_0)} u(x) \leq 2 \cdot 2^\frac{p}{p-1-q} \mathfrak{C}_0  r^{\frac{p}{p-1-q}},
$$
for any normalized solution and any $x_0$ free boundary point. Now, for $\hat{x} \in \partial \{u_{\infty}>0\} \cap \Omega^{\prime}$ we have from uniform convergence that there exist $x_p \to \hat{x}$ with $x_p \in \partial \{u_p>0\} \cap \Omega^{\prime}$. Therefore,
$$
 \displaystyle \sup_{B_r(x_0)} u_{\infty}(x)  = \lim_{p \to \infty} \sup_{B_r(x_p)} u_p(x)\leq
  \lim_{p \to \infty} 2 \cdot 2^\frac{p}{p-1-q} \mathfrak{C}_0 r^{\frac{p}{p-1-q}} = 2\cdot 2^\frac{1}{1-\ell}
 (1-\ell)^\frac{1}{1-\ell} r^\frac{1}{1-\ell},
$$
which is the desired result.
\end{proof}

\begin{proof}[{\bf Proof of Theorem \ref{MThmLim2}}]
 Any sequence of weak solutions $(u_p)_{p\geq 2}$ converges, up to a subsequence, to a limit, $u_{\infty}$,
 uniformly in $\Omega$.
From Theorem \ref{LGR} we have that
$$
   \displaystyle \sup_{B_r(x_0)} \,u(x) \geq  \mathfrak{C}_{0} r^{\frac{p}{p-1-q}} \qquad \text{with} \quad   \mathfrak{C}_{0} \defeq \left[\inf_{\Omega}\lambda_0(x) \frac{ \left(p-1-q \right)^{p} }{ p^{p-1}(pq+N(p-1-q) )}\right]^{\frac{1}{p-1-q}}.
$$
As before for $\hat{x} \in \overline{\{u_{\infty}>0\}} \cap \Omega^{\prime}$ there exist $x_p \to \hat{x}$ with $x_p \in \overline{\{u_p>0\}} \cap \Omega^{\prime}$. Hence we get,
$$
 \displaystyle \sup_{B_r(x_0)} u_{\infty}(x)  = \lim_{p \to \infty} \sup_{B_r(x_p)} u_p(x) \geq (1-\ell)^\frac{1}{1-\ell}
 r^\frac{1}{1-\ell}.
$$
\end{proof}

As a byproduct of the previous estimates we prove that  any limit solution to \eqref{EqLim} is, near the free boundary,  ``trapped'' between the graph of two multiples of $\dist(\cdot, \partial \{u>0\})^{\frac{1}{1-\ell}}$, i.e.,

\begin{corollary} \label{coro62}
 Let $u_{\infty}$ be a uniform limit of $u_p$, solutions to \eqref{Eqp-Lapla}, and $\Omega^{\prime} \Subset \Omega$. Then, for any $x_0 \in \{u_{\infty}>0\} \cap \Omega^{\prime}$ the following estimate holds:
$$
 \mathfrak{C}_1(N, \ell)\dist(x_0, \partial \{u>0\})^{\frac{1}{1-\ell}} \leq  u_{\infty}(x_0) \leq \mathfrak{C}_2(N,\ell)\dist(x_0, \partial \{u>0\})^{\frac{1}{1-\ell}}.
$$
\end{corollary}

\begin{proof}
The upper bound for $u_\infty(x_0)$ follows as in Corollary \ref{DistEst}. For the remaining inequality, let us suppose that such lower bound constant does not exist. Then it should exist a sequence $x_k \in \{u_{\infty}>0\} \cap \Omega^{\prime}$ such that
$$
d_k \defeq \dist(x_k, \partial \{u_{\infty}>0\} \cap \Omega^{\prime}) \to 0 \quad \text{as} \quad k \to \infty \quad \text{and} \quad u(x_k) \leq \frac{d_k^{\frac{1}{1-\ell}}}{k}.
$$
Now, define the auxiliary function $v_k:B_1 \to \R$ by
$$
   v_k(y) \defeq \frac{u_{\infty}(x_k+d_ky)}{d_k^{\frac{1}{1-\ell}}}.
$$
From uniform convergence $v_{k, p} \to v_k$ locally uniformly as $k \to \infty$, where
$$
   v_{k, p}(y) \defeq \frac{u_{p}(x_k+d_ky)}{d_k^{\frac{p}{p-1-q}}}.
$$
Thus,
$$
   -\Delta_p v_{k, p}(y) + \lambda_0(x_k+d_ky)(v_{k, p})_{+}^q(y) \quad \text{in} \quad B_1
$$
in the weak sense. From classical regularity estimates $v_{k, p}$ is H\"{o}lder continuous (see Theorem \ref{MorIneq}). After passing to the limit as $k \to \infty$ we infer that $v_k$ is $\alpha-$H\"{o}lder continuous for any $\alpha \in (0, 1)$.

For this reason, $v_k(y) \leq C|y-0|^{\alpha}+v_k(0) \leq \max\{1, C\}\left(d_k^{\alpha} + \frac{1}{k}\right)\,\,\, \forall\,\, y \in B_1$.
Finally, from the non-degeneracy result, Theorem \ref{MThmLim2}, and the last sentence we obtain
$$
  \displaystyle 0<\left(\frac{1}{2}(1-\ell)\right)^{\frac{1}{1-\ell}} \leq \frac{1}{d_k^{\frac{1}{1-\ell}}}\sup_{B_{\frac{d_k}{2}}(x_k)} u_{\infty}(x) = \sup_{B_{\frac{1}{2}}} v_k(z) \leq \max\{1, C\}\left(d_k^{\alpha} + \frac{1}{k}\right) \to 0 \quad \text{as} \quad k \to \infty.
$$
This contradiction concludes the proof.
\end{proof}

\begin{proof}[{\bf Proof of Theorem \ref{ThmGradLim}}]
From \eqref{EquEstGrad} we have that
$$
   \displaystyle |\nabla u_p(x)| \leq \sup_{B_r(x_p)} |\nabla u_p(x)| \leq 2^{\frac{1+q}{p-1-q}}
   \mathfrak{C}_0r^{\frac{1+q}{p-1-q}} \qquad \text{for} \,\,\,0<r\ll1,
$$
where $x_p \in \partial \{u_p>0\} \cap \Omega^{\prime}$. On the other hand, any sequence of weak solutions $(u_p)_{p\geq 2}$ is bounded in the $W^{1, s}-$topology for all $1<s< \infty$ (and hence we can assume, extracting a subsequence if necessary, that $u_p \to u_\infty$ weakly in $W^{1, 2}$).
Now, for $x_0 \in \partial \{u_{\infty}>0\} \cap \Omega^{\prime}$ there exist $x_p \to x_0$ with $x_p \in \partial \{u_p>0\} \cap \Omega^{\prime}$. Taking into account that $u_p \rightharpoonup \hat{u}$ in $L^2(\Omega)$ and $u_p \to u_{\infty}$ uniformly in $\Omega$ we conclude
$$
  \displaystyle \intav{B_r(x_0)} |\nabla u_{\infty}(x)|dx  \leq \lim_{p\to \infty} \intav{B_r(x_p)} |\nabla u_p(x)|dx \leq 2^{\frac{\ell}{1-\ell}}(1-\ell)^{\frac{1}{1-\ell}} r^{\frac{1}{1-\ell}}.
$$
\end{proof}

\subsection{Further properties for limit solutions}

Now, we present some relevant geometric and measure theoretic properties for limit solutions and their free boundaries.
First, we prove that a version of the well-known Harnack inequality is valid for such limit solutions in balls touching the free boundary. Such a quantitative result is a novelty in the literature with regard to sub-linear limit problems.

\begin{corollary}[{\bf Harnack inequality for limit solutions}]\label{HarIneq}
Let $u_{\infty}$ be a limit solution to \eqref{MThmLim1}, $\Omega^{\prime} \Subset \Omega$ and $x_0 \in \{u_{\infty} > 0\} \cap \Omega^{\prime}$ an interior point such that $d \defeq \dist(x_0, \partial \{u_{\infty}>0\}) \leq \frac{\mathfrak{R}_{\infty}}{2}$, where $\mathfrak{R}_{\infty}$ comes from the limit in \eqref{EqNScond} as $p\to \infty$ in the quantity $\mathfrak{R}_{0}$. Then, there exists $\mathfrak{C}=\mathfrak{C}(N, \ell, [u_{\infty}]_{\text{Lip}(\Omega)})$ such that
\begin{equation} \label{ec.harnack}
   \displaystyle \sup_{\overline{B_{\frac{d}{2}}(x_0)}} u_{\infty}(x) \leq \mathfrak{C}
   \max\left\{\frac{d}{2},\,\,\inf_{\overline{B_{\frac{d}{2}}(x_0)}}  v_{\infty}(x)\right\}.
\end{equation}
\end{corollary}

\begin{proof}
Let $z_1, z_2 \in \overline{B_{\frac{d}{2}}(x_0)}$ be points such that
$$
     \inf_{\overline{B_{\frac{d}{2}}(x_0)}} u_{\infty}(x) = u_{\infty}(z_1) \quad \mbox{and} \quad \sup_{\overline{B_{\frac{d}{2}}(x_0)}} v_{\infty}(x) = u_{\infty}(z_2).
$$
Since $\dist(z_1, \partial \{u_{\infty}>0\}) \geq \frac{d}{2}$, from Corollary \ref{coro62} we get
\begin{equation}\label{eqHar6.1}
   u_{\infty}(z_1) \geq \mathfrak{C}_1(N, \ell) \left(\frac{d}{2}\right)^{\frac{1}{1-\ell}}.
\end{equation}
Moreover, from the Lipschitz regularity for limit solutions we obtain
\begin{equation}\label{eqHar6.2}
   u_{\infty}(z_2) \leq [u_{\infty}]_{\text{Lip}(\Omega)} \left(\frac{d}{2} + v_{\infty}(x_0)\right).
\end{equation}
Now, by choosing $y \in \partial \{u_{\infty} > 0\}$ such that $d=|x_0-y|$, we get as consequence of Theorem \ref{MThmLim3}
\begin{equation}\label{eqHar6.3}
   u_{\infty}(x_0) \leq \sup\limits_{\overline{B_d(y)}} u_{\infty}(x) \leq 2\cdot2^{\frac{1}{1-\ell}} \left(1-\ell\right)^{\frac{1}{1-\ell}}d^{\frac{1}{1-\ell}} = 2\cdot2^{\frac{2}{1-\ell}}\left(1-\ell\right)^{\frac{1}{1-\ell}}\left(\frac{d}{2}\right)^{\frac{1}{1-\ell}}.
\end{equation}
Combining \eqref{eqHar6.1}, \eqref{eqHar6.2} and \eqref{eqHar6.3}, we obtain \eqref{ec.harnack}.
\end{proof}

 \begin{corollary}[{\bf Uniform positive density}]\label{UPDFB}Let $u_{\infty}$ be a limit solution to \eqref{MThmLim1} in $B_1$ and $x_0 \in \partial \{v > 0\} \cap B_{\frac{1}{2}}$ be a free boundary point. Then for any $0<\rho< \frac{1}{2}$,
$$
     \mathcal{L}^N(B_{\rho}(x_0) \cap\{u_{\infty}>0\})\geq \theta \rho^N,
$$
for a constant $\theta>0$ that depends only on the dimension and $\ell$.
\end{corollary}
\begin{proof} Applying Theorem \ref{MThmLim2} there exists a point $\hat{y} \in  \partial B_r(x_0) \cap \{u_{\infty}>0\}$ such that,
\begin{equation}\label{dens}
    v(\hat{y})\geq  (1-\ell)^{\frac{1}{1-\ell}} r^{\frac{1}{1-\ell}}.
\end{equation}
Moreover, from Theorem \ref{MThmLim3} there exists $\kappa>0$ small enough depending only on $\ell$
such that
\begin{equation}\label{inclusion}
    B_{\kappa r}(\hat{y}) \subset  \{u_{\infty}>0\},
\end{equation}
where the constant $\kappa$ is given by
$$
   \kappa \defeq \left(\frac{(1-\ell)^{\frac{1}{1-\ell}}}{14\cdot2^{\frac{1}{1-\ell}}(1-\ell)^{\frac{1}{1-\ell}}}\right)^{1-\ell}.
$$
In fact, if this does not holds, it exists a free boundary point $\hat{z} \in B_{\kappa r}(\hat{y})$. Consequently, from \eqref{dens} we obtain that
$$
   (1-\ell)^{\frac{1}{1-\ell}}r^{\frac{1}{1-\ell}} \leq u_{\infty}(\hat{y}) \leq \sup_{B_{\kappa r}(\hat{z})} u_{\infty}(x) \leq 2\cdot2^{\frac{1}{1-\ell}}(1-\ell)^{\frac{1}{1-\ell}}(\kappa r)^{\frac{1}{1-\ell}} = \frac{1}{7}(1-\ell)^{\frac{1}{1-\ell}}r^{\frac{1}{1-\ell}},
$$
which yields a contradiction. Therefore,
$$
    B_{\kappa r}(\hat{y}) \cap B_r(x_0) \subset  B_r(x_0) \cap \{u_{\infty}>0\},
$$
and hence
$$
     \mathcal{L}^N(B_{\rho}(x_0) \cap\{u_{\infty}>0\})\geq \mathcal{L}^N(B_{\rho}(x_0) \cap B_{\kappa r}(\hat{y}))\geq \theta r^N,
$$
which proves the result.
\end{proof}

\begin{definition}[{\bf $\zeta$-Porous set}] A set $\mathfrak{S} \in \R^N$ is said to be porous with porosity constant $0<\zeta \leq 1$ if there exists an $\mathfrak{R} > 0$ such that for each $x \in \mathfrak{S}$ and $0 < r < \mathfrak{R}$ there exists a point $y$ such that $B_{\zeta r}(y) \subset B_r(x) \setminus \mathfrak{S}$.
\end{definition}

\begin{corollary}[{\bf Porosity of the free boundary}]\label{CorPor} Let $u_{\infty}$ be a limit solution to \eqref{MThmLim1} in $\Omega$. There exists a constant $0<\xi =  \xi(N, \ell) \leq 1$ such that
\begin{equation}\label{eqPor}
       \mathcal{H}^{N-\xi}\left(\partial \{u_{\infty}>0\}\cap B_{\frac{1}{2}}\right)< \infty.
\end{equation}
\end{corollary}

\begin{proof}
Let $\mathfrak{R}>0$ and $x_0\in\Omega$ be such that $\overline{B_{4\mathfrak{R}}(x_0)}\subset \Omega$. We will show that $\partial \{u_{\infty} >0\} \cap B_\mathfrak{R}(x_0)$ is a $\frac{\zeta}{2}$-porous set for a universal constant $0< \zeta \leq 1$. To this end, let $x\in \partial \{u_{\infty} >0\} \cap B_{\mathfrak{R}}(x_0)$. For each $r\in(0, \mathfrak{R})$ we have $\overline{B_r(x)}\subset B_{2\mathfrak{R}}(x_0)\subset\Omega$. Now, let $y\in\partial B_r(x)$ such that $u_{\infty}(y) = \sup\limits_{\partial B_r(x)} u(t)$. By Theorem \ref{MThmLim2}
\begin{equation}\label{5.1}
    u_{\infty}(y)\geq (1-\ell)^{\frac{1}{1-\ell}} r^{\frac{1}{1-\ell}}.
\end{equation}
On the other hand, near the free boundary, from Theorem \ref{MThmLim3} we have
\begin{equation}\label{5.2}
    u_{\infty}(y)\leq 2\cdot2^{\frac{1}{1-\ell}} (1-\ell)^{\frac{1}{1-\ell}}d(y)^{\frac{1}{1-\ell}},
\end{equation}
where $d(y) \defeq \dist(y, \partial \{u_{\infty}>0\} \cap \overline{B_{2\mathfrak{R}}(x_0)})$. From \eqref{5.1} and \eqref{5.2} we get
\begin{equation}\label{5.3}
    d(y)\geq\zeta r
\end{equation}
for a positive constant $0<\zeta \defeq \left(\frac{(1-\ell)^{\frac{1}{1-\ell}}}{2\cdot2^{\frac{1}{1-\ell}}(1-\ell)^{\frac{1}{1-\ell}}}\right)^{1-\ell}<1$.

Now, let $\hat{y}$, in the segment joining  $x$ and $y$, be such that $|y-\hat{y}|=\frac{\zeta r}{2}$, then there holds
\begin{equation}\label{5.4}
   B_{\frac{\zeta}{2}r}(\hat{y})\subset B_{\zeta r}(y)\cap B_r(x).
\end{equation}
Indeed, for each $z\in B_{\frac{\zeta}{2}r}(\hat{y})$
\begin{align*}
   |z-y|&\leq |z-\hat{y}|+|y-\hat{y}|<\frac{\zeta r}{2}+\frac{\zeta r}{2}=\zeta r,\\
   |z-x|&\leq|z-\hat{y}|+\big(|x-y|-|\hat{y}-y|\big)\leq\frac{\zeta r}{2}+\left(r-\frac{\zeta r}{2}\right)=r.
\end{align*}
Then, since by \eqref{5.3} $B_{\zeta r}(y)\subset B_{d(y)}(y)\subset\{u_{\infty}>0\}$, we get $B_{\zeta r}(y)\cap B_r(x)\subset\{u_{\infty}>0\}$, which together with \eqref{5.4} implies that
$$
   B_{\frac{\zeta}{2}r}(\hat{y})\subset B_{\zeta r}(y)\cap B_r(x)\subset B_r(x)\setminus\partial\{u_{\infty}>0\}\subset B_r(x)\setminus \partial\{u_{\infty}>0\} \cap B_{\mathfrak{R}}(x_0).
$$
Therefore, $\partial\{v>0\} \cap B_{\mathfrak{R}}(x_0)$ is a $\frac{\zeta}{2}$-porous set. Finally, the $(N-\xi)$-Hausdorff measure estimates in \eqref{eqPor} follow from \cite{KR}.
\end{proof}

Particularly, Corollary \ref{CorPor} assures that the free boundary $\partial \{u_{\infty}>0\}$ has
$N$-dimensional Lebesgue measure zero.

\begin{theorem}[{\bf Convergence of the free boundaries}]\label{MThmLim5} Let $u_p$ be a sequence of solutions to \eqref{Eqp-Lapla}, $u_{\infty}$ its uniform limit. Then
$$
  \partial \{u_p > 0\} \to \partial \{u_{\infty} > 0\}\quad \mbox{as} \quad  p\to \infty,
$$
locally in the sense of the Hausdorff distance.
\end{theorem}
\begin{proof}
Given  $\delta >0$ let $\mathcal{N}_{\delta}(\mathfrak{S}) \defeq \left\{ x \in \R^N : \dist(x, \mathfrak{S}) < \delta\right\}$ be the $\delta$-neighborhood of a set $\mathfrak{S} \subset \R^N$. We must show that, given $0<\delta \ll 1$ and $p = p(\delta)$ large enough, one obtains
$$
  \partial \{v_p>0\} \subset \mathcal{N}_{\delta} (\partial\{v_{\infty}>0\}) \quad \mbox{and} \quad \partial\{v_{\infty}>0\} \subset \mathcal{N}_{\delta} (\partial \{v_p>0\}).
$$
We proceed with the first inclusion and suppose that it does not hold. Thus, it should exist a point $x_0 \in \partial \{u_p>0\} \cap \left(\Omega \setminus \mathcal{N}_{\delta} (\partial\{u_{\infty}>0\})\right)$. The last sentence implies  $\displaystyle \dist(x_0,  \partial\{u_{\infty}>0\}) \geq \delta$.

Now, if $x_0 \in \{u_{\infty} >0\}$ then by Corollary \ref{coro62} we get
$$
   u_{\infty}(x_0) \geq C(\dist(x_0, \partial \{u_{\infty}>0 \}))^\frac{1}{1-\ell} \geq  C\delta^\frac{1}{1-\ell}.
$$
On the other hand, due to the uniform convergence, for  $p$ large enough
$$
    u_p(x_0) \geq \frac{1}{100}  C\delta^\frac{1}{1-\ell} >0.
$$
However, this contradicts the assumption that $x_0 \in \partial \{u_p>0\}$. Therefore, $u_{\infty}(x_0) = 0$ and then $u_{\infty} \equiv 0$ in $B_{\delta}(x_0)$, which contradicts the strong non-degeneracy property given in Theorem \ref{LGR} since
$$
  \displaystyle \sup_{B_{\frac{\delta}{2}}(x_0)} u_p(x) \geq \mathfrak{c} \left(\frac{\delta}{2}\right)^{\frac{1}{1-\ell}}
$$
from where the inclusion follows.  The second inclusion can be proved similarly and we omit it.
\end{proof}

In  the following result, we analyze the behavior of the coincidence sets for the $p$-variational problem and its corresponding limiting problem. We recall the following notion of limits  sets
$$
  \displaystyle \liminf_{p \to \infty} \,\mathrm{U}_p \defeq \bigcap_{p=1}^{\infty} \bigcup_{k \geq p} \mathrm{U}_k \quad \mbox{and} \quad \limsup_{p \to \infty} \,\mathrm{U}_p \defeq \bigcup_{p=1}^{\infty} \bigcap_{k \geq p} \mathrm{U}_k,
$$
and we say that there exists the limit
$
 \displaystyle \lim_{p \to \infty} \,\mathrm{U}_p$ when $ \displaystyle \liminf_{p \to \infty} \,\mathrm{U}_p =  \limsup_{p \to \infty} \,\mathrm{U}_p$.

\begin{theorem} Let $\mathrm{U}_p \defeq \{u_p = 0\}$ be the null sets of the $p$-dead core problems and $\mathrm{U}_{\infty} \defeq \{u_{\infty} = 0\}$ be the corresponding null set of the
limiting problem.  Assume that along a subsequence $u_p \to u_\infty$. Then, also along the same subsequence, the null sets converge, that is,
$$
 \mathrm{U}_{\infty} = \lim_{p \to \infty} \,\mathrm{U}_p .
$$
\end{theorem}

\begin{proof} We will show that $\displaystyle
 \mathrm{U}_{\infty} \subset \liminf_{p \to \infty} \,\mathrm{U}_p \subset\limsup_{p \to \infty} \,\mathrm{U}_p \subset \mathrm{U}_{\infty}
$.
Given $0< \varepsilon \ll1$ (small enough), consider $\mathcal{V}_{\varepsilon}$ an $\varepsilon$-neighborhood of $\mathrm{U}_{\infty}$. Thus, $\Omega \setminus \mathcal{V}_{\varepsilon} \subset \{u_{\infty}>0\}$ is a closed set. From the continuity of $u_{\infty}$ there exists a positive $\delta = \delta(\varepsilon)$ such that
$
   u_{\infty}(x)> \delta \quad \forall \,\, x \in \Omega \setminus \mathcal{V}_{\varepsilon}$. Moreover, by the uniform convergence (up to a subsequence $u_p \to u_{\infty}$) we obtain that for $p$ large enough
$   u_{p}(x)> \delta \quad \forall \,\, x \in \Omega \setminus \mathcal{V}_{\varepsilon}$. Therefore $\Omega \setminus \mathcal{V}_{\varepsilon} \subset \{u_p>0\}$ from where $\mathrm{U}_p \subset \mathcal{V}_{\varepsilon}$ for every $p\gg 1$.
This implies that $ \limsup_{p \to \infty} \,\mathrm{U}_p \subset \mathcal{V}_{\varepsilon},$ for any $\varepsilon$-neighborhood of $\mathrm{U}_{\infty}$. Hence, we obtain
$\displaystyle \limsup_{p \to \infty} \,\mathrm{U}_p \subset \mathrm{U}_{\infty}$
since $\mathrm{U}_{\infty}$ is a compact set.

Now, let $x_0 \in \mathrm{U}_{\infty}$. We claim that there exists a
sequence $x_p$ with $u_p(x_p) = 0$ such that $x_p\to x_0$.
In fact, from our previous estimates for $u_p$ we have near the free boundary
$$
  \mathfrak{c}_{\sharp} [\dist(x, \partial \{u_p>0\})]^{\frac{p}{p-1-q}} \leq u_p(x).
$$
Hence, in the set $\Omega_p =\{ x\in\Omega \, : \,  \dist(x, \partial \{u_p>0\}) \geq \delta \}$ we have that $u_p \geq C(\delta)$(uniformly in $ p$). Now, as we have that $u_p (x_0) \to u_\infty (x_0) =0$, we obtain that $
\dist(x_0, \partial \{u_p>0\}) \to 0$  as  $p\to \infty$,
and we conclude that, given $\epsilon >0$, for every $p\geq p_0$ there is $x_p \in \mathrm{U}_p$ (that is, with $u_p(x_p) = 0$) such that $\dist(x_p, x_0) <\epsilon$. This shows that $\displaystyle  \mathrm{U}_\infty \subset \liminf_{p \to \infty} \,\mathrm{U}_p$,  as we wanted to prove.
\end{proof}

\begin{definition}[{\bf Reduced free boundary}] The \textit{reduced free boundary} $\mathfrak{F}^{\Omega}_{\text{red}}[u_{\infty}]$ is the set of points $x_0$ for which it holds that, given the half ball $B_r^{+}(x_0) \defeq \{(x-x_0)\cdot\eta\geq 0\}\cap B_r(x_0)$ we get
\begin{equation}\label{eqRedFr}
    \displaystyle \lim_{r\to 0} \frac{\Leb(B_r^{+}(x_0) \triangle \Omega^{+}[u_{\infty}])}{\Leb(B_r(x_0))} = 0.
\end{equation}
This means (cf. \cite[Chapter 3]{Giusti}) that the vector measure $\nabla \chi_{\Omega}(B_r(x_0))$ has a density at the point, i.e., there exists $\eta(x_0)$ (with $|\eta(x_0)|=1$) that fulfills the following
$$
  \displaystyle \lim_{r \to 0} \frac{\nabla \chi_{\Omega}(B_r(x_0))}{|\nabla \chi_{\Omega}(B_r(x_0))|} = \eta(x_0).
$$
\end{definition}

\begin{corollary} If $x_0 \in \mathfrak{F}^{\Omega}_{\text{red}}[u_{\infty}]$ then
$$
   B_r(x_0) \cap \mathfrak{F}^{\Omega}_{\text{red}}[u_{\infty}] \subset \{|(x-x_0)\cdot \eta(x_0)|\leq o(r)\} \quad \text{as} \quad r \to 0^+.
$$
\end{corollary}
\begin{proof}
  If we suppose that $u_{\infty}(x) =  0$ for $(x-x_0)\cdot\eta(x_0) \geq \varepsilon r$, then there exists $\mathfrak{c}_0 > 0$ such that $\Leb(B_{\varepsilon r}(x) \cap \{u_{\infty}=\}) \geq \mathfrak{c}_0\varepsilon r^N$, which implies, according to Corollary \ref{UPDFB}, that
$$
  \displaystyle \liminf_{r \to 0} \frac{\Leb(B_r^{+}(x_0) \triangle \Omega^{+}[u_{\infty}])}{\Leb(B_r(x_0))} \geq \mathfrak{c}_0\varepsilon,
$$
which contradicts \eqref{eqRedFr}.
\end{proof}

We finish this section proving that free boundary points having a tangent ball from inside are regular. To this end, let us introduce the following definition.

\begin{definition}[{\bf Regular points}] A free boundary point $y \in \mathfrak{F}_{\Omega}[u] \defeq \partial \{u>0\}\cap \Omega$ is said to have a \textit{tangent ball from inside} if there exists a ball $\mathcal{B} \subset \Omega^{+}[u] \defeq \{u>0\}\cap \Omega$ such that $y \in  \mathcal{B} \cap \Omega^{+}[u]$. Finally, a free boundary point $y \in \mathfrak{F}_{\Omega}[u]$ is \emph{regular} if $\mathfrak{F}_{\Omega}[u]$ has a tangent hyperplane at $y$.
\end{definition}

\begin{theorem}\label{RegFB} A free boundary point $y \in \mathfrak{F}_{\Omega}[u_{\infty}]$ for a limit solution of the problem \eqref{EqLim} which has a tangent ball from inside is regular.
\end{theorem}

\begin{proof}
The proof follows similarly to \cite[Lemma 11.17]{CafSal}, thus we only sketch the modifications for the reader's convenience. Let us suppose that $B_1(y_1)$ is tangent to $\mathfrak{F}_{\Omega}[u_{\infty}]$ at $y$ and consider the  function
$$
  \Phi(x) = (1-|x-y_1|)^{\frac{1}{1-\ell}}.
$$
From the non-degeneracy, some multiple of $\Phi$, say $\mathfrak{c} \Phi$, is a lower barrier of
$u_{\infty}$ in $B_1(y_1)$. Now, let $\mathfrak{c}_r > 0$ be the supremum of all $\mathfrak{c}$'s such that $u(x) \geq \mathfrak{c} \Phi(x)$ in $B_r(y_1)$.

Notice that such values $\mathfrak{c}_r$ increase with $r$. Hence, by optimal regularity,
$\mathfrak{c}_r$ converges to some constant $\mathfrak{c}_{\infty}$ as $r \to 0$.
According to \cite[Lemma 11.17]{CafSal}, this implies the following asymptotic behavior near the free boundary
$$
  u_{\infty}(x) =  \mathfrak{c}_{\infty} \langle (x-y),\eta(y) \rangle+ o \left(\langle (x-y),\eta(y)\rangle^{\frac{1}{1-\ell}}\right),
$$
where $\eta(y) = y_1-y$. Therefore, the plane orthogonal to $\eta(y)$ is tangent to $\mathfrak{F}_{\Omega}[u_{\infty}]$ and, we conclude that $y$ is a regular point.
\end{proof}

In particular, the previous result reveals that at interior free boundary points verifying the interior ball  condition, limit solutions for the $p-$obstacle problem with zero constraint $q=0$, are regular.

\section{Final comments}\label{SecExam}

Notice that \eqref{EqLim} can be written as a fully nonlinear second order operator as follows
$$
\begin{array}{rcl}
  F_\infty: \R \times \R^N \times \text{Sym}(N) & \longrightarrow & \R \\
  (s, \xi, X) & \mapsto & \max\left\{-\xi^T X\cdot\xi, -|\xi|+ s^{\ell}\right\},
\end{array}
$$
which is non-decreasing in $s$. Moreover, $F_\infty$ is a \textit{degenerate elliptic} operator in the sense that
$$
   F_\infty(s, \xi, X) \leq F_\infty(s, \xi, Y) \quad \text{whenever} \quad Y\leq X \quad \text{in the sense of matrices}.
$$
Nevertheless, $F_\infty$ is not in the framework of \cite[Theorem 3.3]{CIL}.
Hence, we leave open the uniqueness of viscosity solutions to \eqref{EqLim}.  Another open issue is to obtain the optimal regularity of viscosity solutions to \eqref{EqLim}.

In the last part of this paper we include some examples to see what kind of solutions to \eqref{EqLim} one can expect.

\begin{example}[{\bf\bf Radial solutions}] First of all, let us study the following boundary value problem:
\begin{equation}\label{rad eq}
	\left \{
		\begin{array}{rllll}
			-\Delta_p u &=& -\lambda_0 u_{+}^{q}(x) & \text{ in } & B_{R}(x_0), \\
			u(x) &=& \kappa &\text{ on } & \partial B_{R}(x_0),
		\end{array}
	\right.
\end{equation}
where $R, \lambda_0$ and $\kappa$ are a positive constants.

Observe that by the uniqueness of solutions for the Dirichlet problem \eqref{rad eq} and invariance under rotations of the $p-$Laplacian operator, it is easy to see that $u$ must be a radially symmetric function. Hence, let us deal with the following one-dimensional ODE
\begin{equation}\label{rad edo}
			-(|v^{\prime}(t)|^{p-2}v^{\prime}(t))^{\prime} = -\lambda_0 v_{+}^q(t) \quad \text{ in }\, (0, \mathfrak{T}), \qquad v(0)=0 \,\,\,\text{and}\,\,\,v(\mathfrak{T})=\kappa.
\end{equation}
It is straightforward to check that $v(t)=\Theta(1, \lambda_0, p, q) t^{\,\frac{p}{p-1-q}}$ is a solution to \eqref{rad edo}, where
\begin{equation}\label{theta}
	\Theta = \Theta(N, \lambda_0, p, q) \defeq  \left[\lambda_0 \frac{ \left(p-1-q \right)^{p} }{ p^{p-1}(pq+N(p-1-q) )}\right]^{\frac{1}{p-1-q}}
	\quad \mbox{and} \quad
	\mathfrak{T} \defeq \left( \frac{\kappa}{\Theta} \right)^{\frac{p-1-q}{p}}. 	
\end{equation}
Now, in order to characterize the unique solution of \eqref{rad eq}, fix $x_0 \in \mathbb{R}^N$ and $0<\mathfrak{r}_0<R$. We assume the \textit{compatibility condition} for the dead-core problem, namely $R > \mathfrak{T}$. Thus, for $\mathfrak{r}_0 = R-\mathfrak{T}$ the radially symmetric function given by
\begin{equation}\label{RadProf}
  u(x)\defeq \Theta \left[|x-x_0|-R + \left( \frac{1}{\Theta} \right)^{\frac{p-1-q}{p}} \right]_+^{\frac{p}{p-1-q}} = \Theta
  \left[|x-x_0|-\mathfrak{r}_0 \right]_+^{\frac{p}{p-1-q}}
\end{equation}
fulfills \eqref{rad eq} in the weak sense, where $\mathfrak{r}_0 \defeq R - \left( \frac{\kappa}{\Theta} \right)^{\frac{p-1-q}{p}}$. Moreover, the dead-core is given  by $B_{\mathfrak{r}_0}(x_0)$.

Also it is easy to see that the limit radial profile as $p \to \infty$ becomes
\begin{equation} \label{radial}
  u_{\infty}(x)\defeq (1-\ell)^{\frac{1}{1-\ell}}\left(|x-x_0|-\mathfrak{r}_0(\ell) \right)_+^{\frac{1}{1-\ell}},
\end{equation}
which satisfies \eqref{EqLim} in the viscosity sense with $\Omega = B_{R}(x_0)$ the dead-core given by $B_{\mathfrak{r}_0(\ell)}(x_0)$ for $\mathfrak{r}_0(\ell) = R - \frac{1}{1-\ell}$ and $g\equiv \kappa$ on $\partial B_{R}(x_0)$.
\end{example}

\begin{example}[{\bf Aronsson's function}] Let us consider the so-called Aronsson's function, i.e., the two-dimensional function $\mathbb{A}(x, y) \defeq (x^{\frac{4}{3}} - y^{\frac{4}{3}})_{+}$ defined in $\Omega  = B_1(\sqrt{2}, \sqrt{2})$. It is a well-known fact that $\mathbb{A}$ is $\infty-$harmonic, i.e.,
$ -\Delta_{\infty} \mathbb{A} = 0$ in $\Omega\setminus\{x=y\}$. Moreover, it is easy to check that
$$
    |\nabla \mathbb{A}(x, y)| = \frac{4}{3}\sqrt{x^{\frac{2}{3}} + y^{\frac{2}{3}}} \geq \left(x^{\frac{4}{3}} - y^{\frac{4}{3}}\right)^{\frac{1}{4}} = \mathbb{A}^{\frac{1}{4}}(x, y).
$$
Therefore, $\mathbb{A}$ is a viscosity solution to \eqref{EqLim} in $\Omega\setminus\{x=y\}$.
\end{example}

\begin{example} For any fixed $\ell \in \left(0, 1\right)$ consider the function $v:\Omega \subset \R^2 \to \R$ defined by
$$
  \displaystyle v(x, y) \defeq \arctan\left(\frac{y}{x}\right)_{+}
$$
where $\Omega  = B_{\frac{1}{10}}\left(\frac{1}{5},0\right)$. Under such assumptions, it is easy to check that $ -\Delta_{\infty} v(x) = 0$ in  $\Omega\setminus\{y=0\}$. Moreover, one verifies that
$$
   \displaystyle |\nabla v(x, y)| = \frac{1}{\sqrt{x^2 +y^2}} \geq v^{\ell}(x, y) \qquad \text{ in } \Omega \setminus \{y=0\}.
$$
Therefore, $v$ is a viscosity solution to \eqref{EqLim} in $\Omega \setminus \{y=0\}$.
\end{example}

\begin{example}
We remark that the radial limit profile \eqref{radial} is always an $\infty-$subharmonic function whose  $\infty-$Laplacian remains bounded provided that $\ell \in \left[\frac{1}{4}, 1\right)$.

Indeed, from the expression for $u_\infty$ in \eqref{radial} it is easy to check that
$$
   \nabla u_{\infty}(x) =  (1-\ell)^{\frac{l}{1-\ell}}\left(|x-x_0|-\mathfrak{r}_0\right)^{\frac{\ell}{1-\ell}} \frac{x-x_0}{|x-x_0|} \,\,\,\Rightarrow \,\,\,|\nabla u_{\infty}(x)| = u^{\ell}_{\infty}(x).
$$
and
\begin{eqnarray*}
D^2 u_{\infty} \left(x\right) = (1-\ell)^{\frac{1}{1-\ell}}\frac{1}{1-\ell} \left[\frac{\ell}{1-\ell}(|x-x_0|-\mathfrak{r}_0)^{\frac{2\ell-1}{1-\ell}}  \frac{(x-x_0)\otimes (x-x_0)}{|x-x_0|^2}\right. \\
\left. +  |x-x_0|^{\frac{2\ell-1}{1-\ell}} \left( \text{Id}_{N \times N}-\frac{(x-x_0)\otimes(x-x_0)}{|x-x_0|^2}\right)\right],
\end{eqnarray*}
from where we obtain that $-\Delta_{\infty} u_{\infty}(x)= -\left( (1-\ell)^{\frac{1}{1-\ell}}\frac{1}{1-\ell}\right)^3\frac{\ell}{1-\ell}(|x-x_0|-\mathfrak{r}_0)_{+}^{\frac{4\ell-1}{1-\ell}}\leq 0$. On the other hand, if $\ell\in \left(0, \frac{1}{4}\right)$, then $\Delta_\infty u_\infty$ blows up at free boundary points (compare with \cite{RTU} for a free boundary problem of obstacle type driven by the $\infty-$Laplacian).
\end{example}

\subsection*{Acknowledgments}
This work was partially supported by Consejo Nacional de Investigaciones Cient\'{i}ficas y T\'{e}cnicas (CONICET-Argentina).

\end{document}